\documentclass[letterpaper,  11pt]{article}
\usepackage[numbers]{natbib}
\usepackage{thm-restate}
\usepackage[utf8x]{inputenc}
\usepackage[T1]{fontenc}
\usepackage{lmodern}
\usepackage[hidelinks,hypertexnames=false]{hyperref}
\usepackage[numbers]{natbib}
\usepackage[margin=1in]{geometry}
\usepackage{multicol}
\usepackage{booktabs}
\usepackage{amsfonts} \usepackage{amssymb} \usepackage{amsmath}
\usepackage{graphicx} % support the \includegraphics command and options
\usepackage{epstopdf} 

\usepackage{amsthm} \usepackage{enumerate}
\usepackage{url} \usepackage{verbatim} \usepackage{float} 

\usepackage{tabularx,ragged2e,booktabs,caption} \usepackage{makecell}
\usepackage[plain]{algorithm2e}
\usepackage{enumitem} 
\usepackage{caption}

\usepackage{float}
\usepackage{subfig}
\usepackage{makecell}
\usepackage{mathrsfs}
\usepackage{tikz}
\usepackage{multirow}
\usepackage{chngcntr}
\usepackage{mathtools}
\usepackage{cleveref}
\counterwithin{table}{subsection}

\DeclareMathOperator*{\Exp}{Exp}

\DeclareMathOperator*{\bigO}{\ensuremath{\mathcal{O}}}

\DeclareMathOperator*{\trace}{Tr}
\DeclareMathOperator*{\tr}{tr}

\DeclareMathOperator*{\inter}{int}

\DeclareMathOperator*{\proj}{proj}

\newcommand{\coneName}{\mathcal{K}}
\newcommand{\algName}{\mathcal{J}}

\newtheorem{thm}{Theorem}
\numberwithin{thm}{section}

\newtheorem{cor}{Corollary}
\numberwithin{cor}{section}

\newtheorem{lem}{Lemma}
\numberwithin{lem}{section}

\newtheorem{rem}{Remark}

\newtheorem{prop}{Proposition}
\numberwithin{prop}{section}

\newtheorem{defn}{Definition}
\numberwithin{defn}{section}

\numberwithin{ex}{section}

\newtheorem{ass}{Assumption}
\newcommand*{\newton}{d_N}
\newcommand*{\hatnewton}{d_{N, T}}
\title{A geodesic interior-point method for linear optimization over symmetric cones}
\author{Frank Permenter} 
\begin{document}
\maketitle
\newcommand{\sqrtmw}{w^{1/2}}
\newcommand{\sqrtminvw}{w^{-1/2}}
\newcommand{\qup}{q}
\begin{abstract}
We develop a new interior-point method (IPM) for symmetric-cone optimization, a
common generalization of linear, second-order-cone, and semidefinite
programming.  In contrast to classical IPMs, we update iterates with a \emph{geodesic} of the
\emph{cone} instead of the \emph{kernel} of the \emph{linear} constraints.
This approach yields a primal-dual-symmetric, scale-invariant, and
line-search-free algorithm that uses just half the variables of a standard
primal-dual IPM.\@  With elementary arguments, we establish polynomial-time
convergence  matching the standard $\bigO(\sqrt{n})$ bound.  Finally, we prove
global convergence of a long-step variant and provide an implementation that
supports all symmetric cones. For linear programming, our algorithms
reduce to central-path tracking in the log domain.
\end{abstract}

\section*{Introduction}
Let $\algName$ denote a Euclidean Jordan algebra~\cite{faraut1994analysis}
of rank $n$ with multiplication operator  $\circ : \algName \times \algName \rightarrow \algName$,  identity $e \in \algName$,
and trace inner-product $\langle u, v \rangle:= \tr u \circ v$.
This paper considers the following primal-dual pair of linear
optimization problems formulated over the
\emph{cone-of-squares} $\coneName := \left\{ u \circ u : u
\in \algName\right\}$
\begin{align}~\label{sdp:main}
\begin{array}{ll}
  \mbox{ minimize } & \langle s_0, x \rangle     \\
  \mbox{ subject to } &  x \in \coneName \cap (x_0 + \mathcal{L})
\end{array}  \qquad
\begin{array}{ll}
  \mbox{ minimize } & \langle x_0, s \rangle     \\
  \mbox{ subject to } &  s \in \coneName \cap (s_0 + \mathcal{L}^{\perp}),
\end{array}
\end{align}
where $(x, s)$ denotes the primal-dual decision variables, $(x_0, s_0) \in \algName \times \algName$
denotes fixed parameters and $\mathcal{L}\subseteq \algName$ is a linear subspace
with orthogonal complement $\mathcal{L}^{\perp}\subseteq \algName$.  This standard
form~\cite{faybusovich1997linear} generalizes linear, second-order-cone, and semidefinite 
programs~\cite{boyd2009convex}, which typically present  $x_0 + \mathcal{L}$
as the solution set of linear equations
and the dual constraint $s \in \coneName \cap (s_0 + \mathcal{L}^{\perp})$  as a cone inequality.
It is also called a \emph{symmetric cone} program given that 
 $\coneName$ is both \emph{self-dual} and \emph{homogeneous}~\cite{faraut1994analysis}.

This paper contributes to the theory of interior-point methods (IPMs),  
widely used algorithms for solving~\eqref{sdp:main}.
For IPMs, the following assumption is standard.
\begin{ass}\label{ass:slater}
The primal and dual satisfy Slater's condition, i.e.,
$\inter(\coneName) \cap (x_0 + \mathcal{L}) \ne \emptyset$
and $\inter(\coneName) \cap (s_0 + \mathcal{L}^{\perp}) \ne \emptyset$,
where $\inter(\coneName)\subseteq\algName$ denotes the interior of $\coneName$.
\end{ass}
\noindent We make this assumption throughout.

\paragraph{Interior-point methods} 
A pair $(x, s)$ is optimal if it satisfies the constraints of~\eqref{sdp:main}
and the additional \emph{complementary slackness} condition $x \circ s = 0$.
Primal-dual IPMs solve a perturbation of these constraints defined by $\mu > 0$:
  \begin{align}~\label{eq:optcond}
     x \in \coneName \cap (x_0 + \mathcal{L}), \qquad s \in \coneName \cap (s_0 + \mathcal{L}^{\perp}), \qquad x \circ s = \mu e.
  \end{align}
A unique solution $(\hat x(\mu), \hat s(\mu))$ of~\eqref{eq:optcond} exists  for all $\mu > 0$
under \Cref{ass:slater}~\cite[Theorem 2.2]{faybusovich1997euclidean}.
The set of solutions $\{ (\hat x(\mu), \hat s(\mu)) : \mu > 0\}$
is called the \emph{central path}.
Primal-dual IPMs follow the  central path to an optimal solution of~\eqref{sdp:main}, 
i.e., they solve~\eqref{eq:optcond} while gradually reducing $\mu$ to zero.

While there are a variety of primal-dual IPMs (e.g.,~\cite{faybusovich1997linear, schmieta2003extension, nesterov1998primal}),
they share a common feature.  Specifically, when initialized
at feasible points, they all produce iterates $\{x_i, s_i\}^N_{i=1}$ satisfying
\begin{align}~\label{eq:linupdate}
  x_{i+1} - x_i \in \mathcal{L}, \qquad s_{i+1} - s_i \in \mathcal{L}^{\perp},
\end{align}
which implies that $x_i \in x_0 + \mathcal{L}$ 
and $s_i \in s_0 + \mathcal{L}^{\perp}$ for all $i$. These iterations reduce
violation of the complementarity constraint $x_i \circ s_i = \mu e$ and are interleaved with reductions in $\mu$.
For fixed $\mu_0$ and $\mu_f$, IPMs can move from $(\hat x(\mu_0), \hat s(\mu_0))$ 
to $(\hat x(\mu_f), \hat s(\mu_f))$ in $\bigO(\sqrt n)$ iterations,
where $n$ denotes the rank of $\coneName$.

\paragraph{Geodesic interior-point methods}
This paper introduces \emph{geodesic interior-point methods}, a
family of IPMs that views 
$\coneName$ as a Riemannian manifold~\cite{nesterov2008primal,
lim2001riemannian, lee2007metric, faraut1994analysis}.
As indicated by~\eqref{eq:linupdate}, classical IPMs update $(x, s)$ inside  \emph{subspaces} 
that preserve the \emph{affine} constraints.
In contrast, geodesic IPMs will update $(x, s)$ along \emph{geodesic curves} 
that preserve the \emph{complementarity} constraint $x \circ s= \mu e$.
In other words, rather than enforcing~\eqref{eq:linupdate}, they will take
\begin{align}~\label{eq:geoupdate}
  x_{i+1} =  g_{x_i}(t_i) ,  \qquad s_{i+1} =  g_{s_i}(t_i),  
\end{align}
where $t_i \in \mathbb{R}$ is a chosen ``step-size'' and $g_{x_i} : \mathbb{R} \rightarrow \coneName$ and $g_{s_i} :
\mathbb{R} \rightarrow \coneName$ are chosen geodesics satisfying
\begin{align}~\label{eq:geoprop}
  g_{x_i}(0) = x_i \qquad g_{s_i}(0) = s_i, \qquad g_{x_i}(t) \circ g_{s_i}(t) = \mu e  \qquad \forall t.
\end{align}
These iterations will  reduce violation of the affine
constraints $x_i \in x_0 + \mathcal{L}$ and $s_i \in s_0 + \mathcal{L}^{\perp}$ 
and, like~\eqref{eq:linupdate}, will be interleaved with reductions in $\mu$.  Like classical IPMs, we will show that geodesic IPMs can
trace the central path in  $\bigO(\sqrt n)$ iterations,
with essentially identical per-iteration complexity.
We note that while the Riemannian geometry of
$\coneName$ has been used to analyze the central
path~\cite{nesterov2002riemannian, nesterov2008primal,karmarkar1990riemannian}, 
develop gradient methods~\cite{alvarez2004hessian, bomze2019hessian}, and solve non-convex
problems~\cite{absil2009optimization}, to our knowledge no central-path following algorithm
for~\eqref{sdp:main} is based on~\eqref{eq:geoupdate}-\eqref{eq:geoprop}.

\paragraph{Geodesics of symmetric cones}
A geodesic is the shortest path between two points 
as measured by a particular integral cost (made precise in \Cref{sec:geo}).  For this reason, tracing a geodesic curve 
 typically requires solving an ordinary differential equation (ODE)
 that expresses this integral's optimality conditions.
 For symmetric cones, however, geodesics  
 can be expressed in closed form. This in turn provides simple formulae for the update~\eqref{eq:geoupdate}.
 Indeed, for linear programming (LP), i.e., when $\algName = \mathbb{R}^n$
 and $u \circ v$ denotes elementwise multiplication,
 the update~\eqref{eq:geoupdate} will take the form 
 \begin{align}~\label{eq:lpupdate}
   x_{i+1} = x_{i} \circ \exp (t_i d_i), \qquad s_{i+1} = s_i \circ \exp (- t_i d_i)
 \end{align}
 for some $d_i \in \mathbb{R}^n$, where $\exp(u)$ denotes elementwise exponentiation.
For semidefinite programming, 
  i.e., when $\algName$ denotes the symmetric matrices and $U \circ V = \frac{1}{2}(UV+VU)$,
it will take  the form 
 \begin{align}~\label{eq:geosdp}
   X_{i+1} = X^{1/2}_{i}  \exp (t_i D_i) X^{1/2}_{i}, \qquad S_{i+1} = S_i^{1/2} \exp (- t_i D_i) S^{1/2}_i
 \end{align}
 for symmetric $D_i$, where $\exp (U)$ and  $U^{1/2}$ denote the
 matrix exponential and  the symmetric square root.   Similar
 exponential parametrizations hold for arbitrary symmetric cones.
 This will allow us to state and analyze
 algorithms based on~\eqref{eq:geoupdate} using  basic properties of Euclidean Jordan algebras.
 %Typical tools from differential geometry, such as covariant differentiation and parallel
 %transport, will not be needed.
 %In fact, our analysis will only use 
 %We'll see selecting $d_i$ requires similar orthogonal projection operation.
 %Finally, the algorithms will be efficient---both in theory and in practice---thanks
 %in part to these simple formulae.

 %In total, we'll  yields efficient algorithms, both
 %in theory and in practice, for arbitrary symmetric cone programs.
 %Further, these algorithms are scale invariant and primal-dual symmetric~\cite{tunccel1998primal},
 %like the celebrated algorithm of Nesterov~and~Todd. 

% This
% parametrization allows us to state and analyze algorithms for
% selecting $(g_{x_i}, g_{s_i})$ using elementary
% Taylor-series analysis of the exponential map, as oppose
% to more typical machinery of manifold optimization.
\paragraph{Log-domain interpretation}
The LP update~\eqref{eq:lpupdate} is  equivalent to addition
in the log-domain. In fact, for LP, the proposed algorithms essentially reduce to 
Newton's method on a log-domain formulation of the central-path conditions, i.e., to
nonlinear equations $f(z) = 0$  induced by 
\[ 
\sqrt{\mu} \exp(z) \in  x_0 + \mathcal{L}, \qquad  \sqrt{\mu}\exp(-z) \in  s_0 + \mathcal{L}^{\perp}.  
\]
We expand on this log-domain formulation in~\cite{permenter2022LogDomainQP}, but it, surprisingly, seems otherwise unanalyzed.  In fact, to our
knowledge, \emph{all}  interior-point methods for LP---in order to
satisfy~\eqref{eq:linupdate}---operate in the Euclidean space $\algName$ as
opposed to the log-domain; see, e.g.,~\cite{mj1999study, peng2009self}.
%(While previous algorithms~\cite{arora2012multiplicative} for LP   have used
%multiplicative updates like~\eqref{eq:lpupdate}, they are only applied to one
%variable, $x_i$ or $s_i$, and are selected using completely different
%techniques.)

\paragraph{Manifold optimization interpretation}
Our algorithms can be stated using basic concepts from Riemannian
geometry and manifold optimization.  In particular, the key steps reduce to
selection of a \emph{tangent vector}  and evaluation of the \emph{Riemannian exponential map}
associated with $\coneName$.  This viewpoint is crucial to generalizing the
presented techniques to arbitrary convex cones and is discussed
in~\Cref{sec:manifold}.

\paragraph{Relationship with IPMs of Nesterov and Todd}
As we will show, our approach has intimate connections with
that of Nesterov and Todd~\cite{nesterov1998primal}. At a high-level, both yield an algorithm with 
$\bigO(\sqrt n)$ complexity that is \emph{scale-invariant} and \emph{primal-dual
symmetric}~\cite{tunccel1998primal}. At a deeper level, we can interpret
our algorithms as~\cite[Section 6]{nesterov1998primal} modified to perform geodesic updates.
Crucially, this modification removes line searches and computation of a \emph{scaling point},
which requires  eigenvalue decomposition.
It also reduces the number of variables, as we can represent
both $x$ and $s$ using  $w \in \coneName$ satisfying $(x, s) = \sqrt \mu(w, w^{-1})$.
As a trade-off, we must evaluate  the exponential function, but this can
be done using an assortment of techniques~\cite{moler1978nineteen}.  Indeed,
early computational experiments show that our implementation competes with {\tt
sdtp3}~\cite{toh2009sdpt3}, a widely used solver based on the
Nesterov-Todd approach.

\newcommand{\cp}{\mu}
\newcommand{\sqrtcp}{\sqrt{\cp}}
\paragraph{Outline}
This paper is organized as follows. Section~\ref{sec:geo} briefly reviews 
the Riemannian geometry of $\coneName$ and provides a general formula for the geodesic update~\eqref{eq:geoupdate}.  Section~\ref{sec:alg} gives an
IPM  based on geodesic updates and establishes its $\bigO(\sqrt{n})$
complexity, log-domain and manifold optimization interpretations, scale
invariance, and relation to the Nesterov-Todd method.  
We also show that selection of $(g_x, g_s)$, like
 selection of a search direction in classical IPMs,
 reduces to orthogonal projection.
Since this procedure conservatively tracks the central path, we
refer to it as our \emph{short-step} algorithm~\cite{wright1997primal}.  In Section~\ref{sec:geodesic},
we study connections between geodesic distance and symmetrized Kullback-Leibler
divergence, proving key results invoked in our short-step analysis.  Leveraging
this study, we describe a less conservative \emph{long-step} algorithm in
Section~\ref{sec:prac} and prove its global convergence and scale invariance; we also
discuss efficient computation of geodesic updates,
construction of feasible points, and other implementation issues.
Finally, Section~\ref{sec:comp} contains computational results and links to an implementation.

\section[]{Geodesic updates for symmetric cones}~\label{sec:geo}
The interior of a symmetric cone $\coneName$, denoted $\inter \coneName$,
can be viewed as a Riemannian manifold by equipping
each $u \in \inter\coneName$ with a local norm $\|\cdot \|_u$
using the \emph{quadratic
representation} $Q(u) : \algName \rightarrow \algName$, the self-adjoint, linear
map induced by $u \in \algName$ via the 
relation $Q(u)v := 2 u\circ (u \circ v) - (u\circ u) \circ v$.  
For $u \in \inter\coneName$, the map $Q(u)$
is also positive definite, leading to the definition
 $\|v\|_u := \|Q(u)^{-1/2} v\|$, where
 $\|w\|:= \sqrt{\langle w, w \rangle}$. 
The local norm $\|\cdot\|_u$ in turn induces an arc-length $L(\gamma)$
for smooth curves $\gamma : [0, 1]  \rightarrow \inter \coneName$
via
\[
  L(\gamma) := \int^1_0 \|\gamma'(t)\|_{\gamma(t)} dt.
\]
We note that this Riemannian geometry is studied 
by~\cite[Chapter 6]{bhatia2009positive} for the cone of positive definite
matrices and by~\cite{lawson2001geometric,
lee2007metric, lim2001riemannian, faraut1994analysis} for general
symmetric cones.

For $u, v \in \inter\coneName$, let $\delta(u, v)$ denote the infimum of $L(\gamma)$
over smooth curves $\gamma(t)$ satisfying $\gamma(0) = u$ 
and $\gamma(1) = v$.  A curve of length $\delta(u, v)$
connecting $u$ and $v$ is called a \emph{geodesic}.
Useful properties are collected below, including explicit formulae
for $\delta(u, v)$  and geodesic curves.
These formulae employ the square root $u^{1/2}$ and inversion $u^{-1}$
operations of the algebra $\algName$, as well as its log and exponential functions.
\begin{lem}[e.g.,~\cite{lawson2001geometric, lee2007metric}]~\label{lem:GeoProp}
  The following statements hold:
  \begin{enumerate}[label= (\alph*)]
    \item~\label{lem:GeoProp:metric}  $\delta(u, v)$ is a metric on $\inter\coneName$.
    \item~\label{lem:GeoProp:normal} Given $u, v \in \inter \coneName$, let  
    $d :=\log Q(u^{-1/2}) v$
      and $g(t) := Q(u^{1/2}) \exp (t  d)$.  The curve
        $g(t)$ is a geodesic from $u$ to $v$, i.e.,
      \[
      g(0) = u, \qquad g(1) = v, \qquad L(g) = \delta(u, v).
      \]
      Further, $\delta(u, v) = \|d\|$.
    \item~\label{lem:GeoProp:scale}  $\delta(u, v) =  \delta(T u, T v)$ for all $u, v \in \inter \coneName$
        and for any automorphism $T$ of $\coneName$,
        i.e., for any invertible, linear map $T: \algName \rightarrow \algName$ satisfying $\{ T z : z \in \coneName\} = \coneName$.

    \item~\label{lem:GeoProp:inverse}  $\delta(u^{-1}, v^{-1}) =  \delta( u, v)$ for all $u, v \in \inter \coneName$
  \end{enumerate}
\end{lem}
\noindent In light of~\ref{lem:GeoProp:metric}, the function $\delta(u, v)$ is called \emph{geodesic distance}.
The vector $d$ in~\ref{lem:GeoProp:normal}  
denotes \emph{normal coordinates} of $v$ at the point $u$.
In light of~\ref{lem:GeoProp:scale}, the inner-product
$\langle v, w \rangle_u := \langle v,  Q(u)^{-1} w \rangle$
associated with $\|\cdot\|_u$ is called a \emph{scale-invariant} or
\emph{affine-invariant} metric for $\coneName$.  
\Cref{lem:GeoProp:inverse} shows inversion is an \emph{isometry}.
Note that $g(0) = u$ and $g(1) = v$ in~\ref{lem:GeoProp:normal}
is immediate from the identities $Q(u^{1/2}) = Q(u)^{1/2}$,
$Q(u^{-1}) = Q(u)^{-1}$, and $Q(u^{1/2}) e = u$; see \Cref{sec:appendix}.

\subsection{Complementary geodesics}
Given $x, s \in \inter\coneName$ satisfying $x \circ s = \mu e$, 
we wish to parametrize geodesics $g_x$ and $g_s$ starting at $x$ and $s$
that satisfy $g_x(t) \circ g_s(t) = \mu e$ for all $t$.
Combining \Cref{lem:GeoProp} with properties of the quadratic
representation $Q(u)$ provides a parametrization in terms of 
$d \in \algName$. We also
express $(x, s)$ using the point $w \in \coneName$ satisfying $(x, s) = \sqrt\mu (w, w^{-1})$.
%\begin{lem}[e.g.,]~\label{lem:GeoCurveProp}
%  For $w \in \inter \coneName$ and $d \in \algName$
%  let $g(t) := Q(w^{1/2}) \exp(td)$. Then,
%  \begin{enumerate}[label= (\alph*)]
%    \item $g(t)$ is the unique geodesic connecting $w$ and $g(1)$.
%    \item $[g(t)]^{-1}$ is the unique geodesic connecting $w^{-1}$ and $g(1)^{-1}$
%      and satisfies $[g(t)]^{-1} =  Q(w^{-1/2}) \exp -td$.
%  \end{enumerate}
%\end{lem}
\begin{prop}~\label{prop:paramGeo}
  For $d \in \algName$, $w \in \inter\coneName$ and $\mu > 0$, let $(x, s) = \sqrt \mu(w, w^{-1})$
  and let
  \begin{align}\label{eq:geoupchar}
    g_x(t) = \sqrt \mu Q(w^{1/2}) \exp (t d), \qquad g_s(t) = \sqrt \mu Q(w^{-1/2}) \exp (-t d).
  \end{align}
  Then, $g_x(t)$ and $g_s(t)$ are geodesics satisfying $g_x(0) = x$,
  $g_s(0) = s$, and $g_x(t) \circ g_s(t) = \mu e$ for all $t \in \mathbb{R}$.
  \begin{proof}

    The condition $g_x(t) \circ g_s(t) = \mu e$  holds from the identity 
    $[Q(u)v]^{-1} = Q(u^{-1}) v^{-1}$, whereas $g_x(0) = x$ and $g_s(0) = s$
    hold from the identities $\exp(0) = e$ and $Q(u^{1/2})e = u$;  see~\Cref{sec:appendix}.
    Finally, that $g_x(t)$ and $g_s(t)$ are geodesic
    follows from~\Cref{lem:GeoProp}~\ref{lem:GeoProp:normal}
    and the identity $Q( (\sqrt c u)^{1/2}) = \sqrt c Q(u^{1/2})$.

  \end{proof}
\end{prop}
\subsection{Newton direction}
\Cref{prop:paramGeo} shows that the geodesic update 
of $(x, s) = \sqrt\mu(w, w^{-1})$ introduced by~\eqref{eq:geoupdate}
is performed by selecting
$d \in \algName$ and evaluating~\eqref{eq:geoupchar} at some $t$.
Since our goal is to decrease violation of the affine constraints $x \in x_0 + \mathcal{L}$
and $s \in s_0 + \mathcal{L}^{\perp}$, a natural choice for $d$ is the  \emph{Newton direction}, 
which we define  by 
substituting~\eqref{eq:geoupchar} into the central-path conditions~\eqref{eq:optcond}
with the linearizations $\exp(d) \approx e + d$ and $\exp(-d)\approx e - d$.
\begin{defn} (Newton Direction)~\label{defn:NewtonDir}
  For $w \in  \inter \coneName$ and $\cp > 0$, the \emph{Newton direction}
  $\newton(w, \cp)$ is the unique $d \in \algName$ satisfying
  \begin{align*} 
    Q(w^{1/2})(e + d)  \in \frac{1}{\sqrtcp} x_0 + \mathcal{L}, \qquad Q(w^{-1/2})(e -d) \in  \frac{1}{\sqrtcp} s_0 + \mathcal{L}^{\perp}.
\end{align*}
\end{defn}
\noindent Uniqueness of $\newton(w, \cp)$  is proven later by \Cref{prop:newProj},
but essentially follows from invertibility of $Q(w^{1/2})$ and $Q(w^{-1/2})$.
Geodesic updates using $\newton(w, \cp)$  are the basis of algorithms
given in~\Cref{sec:alg} and \Cref{sec:prac}.

%\noindent We note that~\cite[Chapter 6]{bhatia2009positive} is an excellent
%reference for this Riemannian geometry when $\algName$ is the Jordan algebra
%of symmetric matrices.
 %, and differential length simplifies to $\|\gamma^{-1/2}(t) \gamma'(t)\gamma^{-1/2}(t)\|$.
%\noindent The metric property of $d$ will allow us to analyze how updates in $\mu$ affect
%convergence of Newton's method, and the formulae for $g(t)$ and $g(t)^{-1}$ will
%allow us to maintain the complementarity relation $x \circ s$ by updating
%these variables along $g(t)$ and $[g(t)]^{-1}$.

\section{Short-step algorithm}\label{sec:alg}
\begin{figure}
  \hspace{-.6cm}
  \begin{minipage}{.45\linewidth}
  \begin{algorithm}[H]
  \SetAlgoLined\DontPrintSemicolon{}
  \SetKwFunction{algo}{shortstep}\SetKwFunction{proc}{Center}
  \SetKwProg{myalg}{Procedure}{}{}
  \myalg{\algo{$w_0, \mu_0, \mu_f$}}{% chktex 
  \vspace{.1cm}
  $w \leftarrow w_0$, $\cp \leftarrow \cp_0$\\
  \While{$\mu  > \mu_f$} 
  {% chktex
    $\cp \leftarrow \frac{1}{k}\cp $ \\
    \For{$i = 1, 2, \ldots, m$}
    {% chktex
      $d \leftarrow  \newton(w,  \cp)$ \\
      $w \leftarrow  Q(w^{1/2}) \exp(d)$ \\
    }
 }
    \Return{$(w, \mu)$}  \\
  }{} 
  %\caption{Path-following algorithm with parameters $(k, m)$
\end{algorithm}
  \end{minipage}
  %\hspace{-1cm}
  \begin{minipage}[t]{.45\linewidth}
    \small
  \vspace{-2.3cm}
    \begin{tabular}[t]{lp{5.1cm}c}
         $\coneName$&  Definition & rank  \\
        \toprule
         $\mathbb{R}^n_{+}$ & $\{ x \in \mathbb{R}^n : x_i \ge 0\}$ & $n$ \\
        $\mathbb{S}^{n}_+$ & $\{ X^2 : X \in \mathbb{R}^{n \times n}, X= X^T \}$  & $n$  \\
         $\mathbb{L}^{m+1}$ & $\{ (x_0, x_1) \in \mathbb{R}  \times \mathbb{R}^m : x_0 \ge \|x_1\|  \}$ & $2$ 
      \end{tabular}

     \vspace{.5cm} 
    \begin{tabular}[t]{lp{3.1cm}l}
         $\coneName$&  $\exp(d)$ & $Q(w^{1/2})\exp(d)$  \\
        \toprule
         $\mathbb{R}^n_{+}$ &  element-wise exp. & $\exp(\log w + d)$\\
         $\mathbb{S}^n_{+}$ &  matrix exponential& $W^{1/2}\exp(D) W^{1/2}$ \\
        $\mathbb{L}^{m+1}$ & replace eigenvalues \newline with $\exp(d_0\pm\|d_1\|)$  & $(2 z z^T - (\det z) R) \exp(d)$
      \end{tabular}
  \end{minipage}
  \caption{Short-step  algorithm (left) with parameters $(k, m)$  and
  implementation details (right) for linear programs $(\mathbb{R}^n_{+})$,
  second-order-cone programs $(\mathbb{L}^{m+1})$, and semidefinite
  programs $(\mathbb{S}^n_{+})$.  In the  $\mathbb{L}^{m+1}$ row, the map $R$
  denotes $(u_0, u_1) \mapsto (u_0, -u_1)$, while $z= w^{1/2}$ and $\det z =
  z^2_0 - \|z_1\|^2$. }~\label{alg:barrier}
    %// We use the formula from Example 11.12 of "Formally Real Jordan Algebras
    %// and Their Applications to Optimization"  by Alizadeh, which states the quadratic
    %// representation of x equals the linear map
    %//                          2xx' - (det x) * R
    %// where R is the reflection operator R = diag(1, -1, ..., -1) and det x is the determinate
    %// of x = (x0, x1), i.e., det x = x0^2 - |x1|^2.
    %Real det_x = x(0) * x(0) - x.tail(order - 1).squaredNorm();
    %EigenType z = det_x * y;
    %z(0) *= -1;
    %return (2 * x.dot(y)) * x + z;
\end{figure}
We give a procedure {\tt shortstep} (\Cref{alg:barrier}) for tracking the central path
that employs the geodesic updates described by Proposition~\ref{prop:paramGeo}.
Per this proposition, it updates $w \in \inter \coneName$ satisfying
\begin{align}\label{eq:xfroms}
    x  =  \sqrtcp w, \qquad  s = \sqrtcp  w^{-1}
\end{align}
 via $w \leftarrow  Q(w^{1/2}) \exp(t d)$, or, equivalently, 
via $w^{-1} \leftarrow Q(w^{-1/2}) \exp(- td)$. At each iteration, it sets $d$
equal to the Newton direction  $\newton(w, \mu)$ and the step-size  $t$  equal to one, i.e.,
it performs a full \emph{Newton step}.
By construction, each iterate $w$
induces via~\eqref{eq:xfroms} 
variables $x$ and $s$ satisfying $x \circ s = \mu e$. 
Each Newton step   in turn aims to reduce
the violation of the affine constraints $x \in x_0 + \mathcal{L}$  and
$s \in s_0 + \mathcal{L}^{\perp}$.

 The inputs are  an initial $w_0 \in \inter
\coneName$ and centering parameters $\mu_0, \mu_f \in \mathbb{R}$ satisfying $\mu_0 > \mu_f > 0$.  The output is an
approximation of the centered point $\hat w(\mu)$ for $\mu \le \mu_f$, where
$\hat w(\mu)$ denotes the unique point satisfying $\sqrt\mu (\hat w(\mu), \hat w(\mu)^{-1}) = (\hat x(\mu), \hat s(\mu))$ for $(\hat x(\mu),
\hat s(\mu))$ on the central path.  
 Behavior depends on a parameter  
$k$ that controls how much $\mu$ decreases at each \emph{outer iteration}
and a parameter $m$ that denotes the number of \emph{inner iterations}.
Like short-step IPMs~\cite{wright1997primal}, our analysis will 
 choose $k$ conservatively   and assume that $w_0 = \hat w(\mu_0)$.
A more aggressive algorithm that supports arbitrary initialization by using 
damped updates $(t < 1)$ appears in Section~\ref{sec:prac}.

To establish convergence results, we need two lemmas whose proofs we postpone to
Section~\ref{sec:geodesic}. They employ the function 
$\qup : \mathbb{R} \rightarrow \mathbb{R}_+$ and
its nonnegative inverse $\qup^{-1} : \mathbb{R}_+ \rightarrow \mathbb{R}_+$
defined via:
 \[
  \qup(u) := 2(\cosh(u) - 1), \qquad \qup^{-1}(u)  :=  \cosh^{-1} (1 + \frac{1}{2}u ).
\]
In passing, we observe that $q(u) \ge u^2$ for all $u \in \mathbb{R}$ and $\sqrt u \ge q^{-1}(u)$ 
for all $u \ge 0$.
The first lemma  bounds the geodesic distance between two centered points
$\hat w(\cp_0)$ and $\hat w(\cp_1)$ using the rank  of $\coneName$ (denoted by $n$) and the
ratio $k$ of the centering parameters.
\begin{restatable}[$\mu$-update]{lem}{centeringUpdate}\label{lem:centeringUpdate}
  Let $\mu, k > 0$. Then,
  $\frac{1}{n} \delta\left(\hat w(\mu), \hat w( \frac{1}{k} \mu)\right)^2 \le  \qup(\frac{1}{2}\log k)$.
\end{restatable}
\noindent The second lemma establishes a region of quadratic convergence 
of the sequence $w_0, w_1, \ldots, w_m$ generated by Newton steps (inner iterations). 
\begin{restatable}[Centering]{lem}{newtonUpdate}\label{lem:newtonUpdate}
  For $\cp > 0$ and $w_0 \in \inter \coneName$, recursively define $w_i$
  via the iterations $w_{i+1} = Q(w_i^{1/2}) \exp(\newton(w_i, \mu))$.
  If $\delta\left(w_0, \hat w(\mu)\right) \le
  \qup^{-1}(\beta)$ for $0 \le \beta \le \frac{1}{2}$, then $\delta(w_i, \hat w(\cp))^2 \le  \beta^{2^i}$.
\end{restatable}
\noindent Our main result follows from these lemmas,
the triangle inequality for geodesic distance $\delta$,
and an inequality relating the scalar functions $q^{-1}$ and $\sqrt{u}$.
\begin{thm}[Main Result]~\label{thm:barrier}
  Let  {\tt shortstep} (\Cref{alg:barrier}) have parameters $(k, m)$ that satisfy, for
  some  $\frac{1}{2} \ge  \beta > 0$ and $\qup^{-1}(\beta) > \epsilon > 0$, the conditions
  \begin{align}\label{eq:mainass}
    \beta^{2^m}  \le \epsilon^2, \qquad \frac{1}{2} \log k = \qup^{-1}(\frac{1}{n} \zeta^2),
  \end{align}
  where $\zeta := \qup^{-1}(\beta) - \epsilon$.  Then, 
  the following statements hold for {\tt shortstep} given input $(\hat w(\cp_0), \cp_0, \cp_f)$:
  \begin{enumerate}[label= (\alph*)]
      \item At most  $m \lceil c^{-1} \sqrt{n} \log \frac{\cp_0}{\cp_f} \rceil$ Newton steps execute,
        where $c := 2\qup^{-1}(\zeta^2)$.
      \item  The output $(w, \mu)$ satisfies $\delta(w, \hat w(\cp)) \le \epsilon$ and $\cp \le \cp_f$.
        Further,
        \[
          \delta( \sqrtcp w , \hat x(\cp) ) \le \epsilon, \qquad \delta(  \sqrtcp w^{-1}  , \hat s(\cp) ) \le \epsilon,
        \]
        where $(\hat x(\cp), \hat s(\cp))$ denotes the solution to the central-path conditions~\eqref{eq:optcond}.
    \end{enumerate}
  \begin{proof}
    Let $\gamma = (\log k)^{-1} \log \frac{\cp_0}{\cp_f} $.
    The number of outer iterations is at most $\lceil \gamma \rceil$.  
    To upper bound $(\log k)^{-1}$, we first note that for any $a, b \in
    \mathbb{R}$ satisfying $0 \le a \le b$, 
    \[
      \qup^{-1}(a) \ge \frac{\qup^{-1}(b)}{\sqrt b} \sqrt{a},
    \]
    since $\qup^{-1}(u)/\sqrt u$ is a decreasing function.
    Setting  $a = \frac{1}{n}\zeta^2$ and $b = \zeta^2$ and using~\eqref{eq:mainass} gives
    \[
      \frac{1}{2}  \log k =  \qup^{-1}(\frac{1}{n} \zeta^2)  \ge  \frac{\qup^{-1}(\zeta^2)}{\zeta}  \frac{\zeta}{\sqrt n} = \frac{\qup^{-1}(\zeta^2)}{\sqrt n}.
    \]
      Hence, $(\log k)^{-1} \le  c^{-1} \sqrt n$ for $c = 2 \qup^{-1}(\zeta^2)$, proving the first statement.

    We prove the next statement using induction on outer iterations, observing first that
    \begin{align}\label{eq:proofMainCent}
      \delta(\hat w(\mu), \hat w( k^{-1} \mu))^2 \le n [ q(\frac{1}{2} \log k)] = n \frac{1}{n}\zeta^2 =  (\qup^{-1}(\beta) - \epsilon)^2
    \end{align}
    by \Cref{lem:centeringUpdate} and our choice of $k$.
    Now let $w_i^{\mu}$ denote $w$ at the end of inner iteration $i$ for the current $\cp$
    and let $\mu' := k^{-1} \mu$.
    Make the inductive hypothesis that $\delta(w_m^{\cp},  \hat w(\cp))  \le \epsilon$.
    Then,
    \begin{align*}
      \delta(w_m^{\cp},  \hat w(\mu') )  \le \delta(w_m^{\cp},  \hat w(\cp)) + \delta(\hat w(\mu), \hat w(\mu'))  \le \epsilon  +\qup^{-1}(\beta) - \epsilon = \qup^{-1}(\beta),
    \end{align*}
    where the second inequality follows from~\eqref{eq:proofMainCent} 
    and the first is the triangle inequality (\Cref{lem:GeoProp}\ref{lem:GeoProp:metric}).
    This allows us to update $\mu$ to $\mu'$, restart inner iterations at $w^{\mu'}_0 =w^{\mu}_m$,  
    and use \Cref{lem:newtonUpdate}  to conclude that
    \[
      \delta(w_m^{\mu'},  \hat w( \mu' ))^2 \le  \beta^{2^m}  \le \epsilon^2,
    \]
    where the second inequality follows from our choice of $ m$.
    Hence we've show if  $\delta(w_m^{\cp},  \hat w(\cp))  \le \epsilon$,
    then $\delta(w_m^{\mu'},  \hat w(\mu'))  \le \epsilon$.
    The base case holds by identical argument using the assumption that $w_0 = \hat w(\mu_0)$.
    Hence, $\delta(w,  \hat w(\cp))  \le \epsilon$ holds at termination. That
       $\delta( \sqrtcp w , \sqrtcp \hat w(\cp) ) \le \epsilon$ 
      and $\delta( \sqrtcp w^{-1} , \sqrtcp \hat w^{-1}(\cp) ) \le \epsilon$
      follows from the invariance of $\delta$ under rescaling and inversion; see \Cref{lem:GeoProp}\ref{lem:GeoProp:scale}-\ref{lem:GeoProp:inverse}.
  \end{proof}
\end{thm}
The remainder of this section gives other properties of {\tt shortstep},
namely,  log-domain and manifold optimization interpretations, an orthogonal decomposition of the 
Newton direction, and scale invariance. We also discuss connections with an
algorithm of Nesterov~and~Todd.

\subsection{Log-domain interpretation}
Suppose that~\eqref{sdp:main} is a primal-dual pair of \emph{linear
programs}, i.e., that $\coneName = \mathbb{R}^n_{+}$. Under this assumption, the algebra 
$\algName$ is associative.  Hence,  geodesic distance simplifies to $\delta(u,
v) = \| \log u - \log v\|$,
and the geodesic update $w \leftarrow
Q(w^{1/2}) \exp(d)$ satisfies
\begin{align}\label{eq:loggeo}
  \log  \left( Q(w^{1/2}) \exp(d) \right)  = \log (w \circ \exp (d)) =  \log (w) + d,
\end{align}
i.e., it reduces to addition in the log domain. 
The Newton direction $\newton(w, \mu)$ also has a log-domain interpretation: it
is precisely the direction one obtains by linearizing $x(z) := \sqrt \mu \exp
z$ and $s(z) := \sqrt \mu \exp(-z)$ at $z = \log w$ and substituting into the
central-path conditions~\eqref{eq:optcond}.
\begin{prop}
  Let $\algName$ be associative. For $\mu >0$ and $w \in \inter \coneName$, let $d = \newton(w, \mu)$.
  Then,
  \[
    \exp(z) + J(z)d \in \frac{1}{\sqrtcp} x_0 + \mathcal{L}, \qquad  \exp(-z)- J(-z)d \in \frac{1}{\sqrtcp} s_0 + \mathcal{L}^{\perp},
  \]
  where $z = \log w$ and $J(z) : \algName \rightarrow \algName$ is the Jacobian of $\exp(z)$.
  \begin{proof}
    Under our associativity assumption, we observe that $J(z) d = \exp(z) \circ d$ and
    \[ 
    Q(w^{1/2})(e + d) = w + w \circ d, \qquad  Q(w^{-1/2})(e - d) = w^{-1} - w^{-1} \circ d.
    \]
    Substituting $w = \exp(z)$ and $w^{-1} = \exp(-z)$ and using Definition~\ref{defn:NewtonDir} proves the claim.
  \end{proof}
\end{prop}
\noindent In total, we can reinterpret the inner iterations of {\tt shortstep}
as simply Newton's method applied to the central-path conditions in the log domain.  We elaborate on this 
interpretation (and extend it to quadratic optimization) in the paper~\cite{permenter2022LogDomainQP}.

Observe that when $\algName$ is \emph{not} associative, this interpretation
fails because the identity~\eqref{eq:loggeo} fails. For semidefinite
programming, failure of~\eqref{eq:loggeo} reduces to the fact that
for matrices $W \succ 0$ and $D$, 
\[
  \log \left( W^{1/2} \exp(D) W^{1/2} \right) \ne \log(W) + D,
\]
since, in general, $\exp(A+B) \ne \exp(A) \exp(B)$ for the matrix exponential.

\subsection{Manifold optimization interpretation}\label{sec:manifold}
Geodesic updates can be alternatively described using the \emph{Riemannian exponential map}
  of  $\inter \coneName$.
This function, denoted $\Exp_{u} : \algName \rightarrow \inter\coneName$,
 maps  \emph{tangent vectors} $v \in \algName$ to points on geodesics passing through $u \in \inter\coneName$.
Precisely, $\Exp_{u}(v) = g(1)$, where $g : [0, 1] \rightarrow \inter \coneName$ 
is the geodesic satisfying
\[
  g(0) = u, \qquad \dot g(0) = v,
\]
where $\dot g(t) := \frac{d}{dt} g(t)$.
For a general manifold $\mathcal{M} \subseteq \mathbb{R}^m$, evaluating this map requires
solving a system of 2nd-order ODEs of the form 
\begin{align}~\label{eq:geodesic}
\ddot g_k  + \sum^m_{i=1} \sum^m_{j=1} \Gamma^k_{ij} \dot g_i \dot  g_j = 0, \;\; k = 1, 2, \ldots, m,
\end{align}
where $\Gamma^k_{ij} \in \mathbb{R}$ are the \emph{Christoffel symbols} of $\mathcal{M}$.
For symmetric cones, however, $\Exp_{u}$ has an explicit formula involving the exponential map  of the \emph{algebra} $\algName$:
\[
  \Exp_u(v)  = Q(u^{1/2}) \exp (Q(u^{-1/2})  v).
\]
Further, we can  express geodesic
updates (\Cref{prop:paramGeo}) of the primal-dual variables $(x, s) = \sqrt\mu(w, w^{-1})$
using  $\Exp_x$ and $\Exp_s$.
\begin{prop}
For $w \in \inter\coneName$ and $\mu > 0$,
let $d = \newton(w, \mu)$ and define
\begin{align}\label{eq:tanDef}
  x := \sqrt\mu w, \;\; s := \sqrt\mu w^{-1}, \;\;
  d_x := \sqrt\mu Q(w^{1/2}) d, \;\;
  d_s := -\sqrt\mu Q(w^{-1/2})d.
\end{align}
The following statements hold.
  \begin{itemize}
    \item $\Exp_x (d_x) = \sqrt\mu Q(w^{1/2}) \exp(d)$.
    \item $\Exp_s (d_s) = \sqrt\mu Q(w^{-1/2}) \exp(-d)$.
    \item $\Exp_x (d_x) \circ  \Exp_s (d_s) = \mu e$.
  \end{itemize}
  \begin{proof}
    We first observe that
    \[
      \sqrt\mu Q(w^{1/2})  =  Q(\mu^{1/4} w^{1/2}) =   Q( {(\sqrt\mu w)}^{1/2}) = Q(x^{1/2}),
    \]
    which implies that $Q(x^{-1/2}) = \frac{1}{\sqrt\mu} Q(w^{-1/2})$.
    Evaluating $\Exp_x$ at  $d_x  := \sqrt\mu Q(w^{1/2}) d $ 
    yields
    \[
      \Exp_x (d_x)   = Q(x^{1/2}) \exp (Q(x^{-1/2}) \sqrt\mu Q(w^{1/2}) d)
      = \sqrt\mu Q(w^{1/2}) \exp(d).
    \]
    The second
    statement follows by identical argument.
    The third follows using the first two statements
    and the identity $(Q(w^{1/2}) \exp(d))^{-1}= Q(w^{-1/2}) \exp(-d)$; see \Cref{lem:quad}.
  \end{proof}
\end{prop}

Recalling the definition of the local norm 
      $\|v\|_u := \|Q(u)^{-1/2} v\|$,
 we can also characterize the tangent vectors $d_x$ and $d_s$
without reference to $\newton(w, \mu)$.

\begin{prop}~\label{prop:manifoldTangent}
The tangent vectors $d_x$ and $d_s$
in~\eqref{eq:tanDef} are the unique points in $\algName$ satisfying 
  \[
    x + d_x \in x_0 + \mathcal{L}, \quad
    s + d_s \in s_0 + \mathcal{L}^{\perp},   \quad
    d_s = -\mu Q(x)^{-1} d_x.
  \]
  Further, $\|d\| = \|d_x\|_x$ and $\|d\| = \|d_s\|_s$.
    \begin{proof}
    The first two conditions are immediate
     from definition of  $(x, d_x)$, $(s, d_s)$,
      and the Newton direction $\newton(w, \mu)$.
    To see that $d_s = -\mu Q(x)^{-1} d_x$,
      observe first that
    \[
      \sqrt\mu d =  Q(w^{-1/2})  d_x, \qquad -\sqrt\mu d =  Q(w^{1/2})  d_s.
    \]
    Hence, $-d_s =  Q(w^{-1}) d_x =  Q( \sqrt\mu  x^{-1}) d_x =  \mu Q(x^{-1}) d_x$.
    Uniqueness follows from the fact $Q(x)$ is invertible and the fact
      the first two conditions are equivalent to $\dim \algName$
      linearly independent constraints.

    For the last statement, we have that
    \[
      \|d_x\|^2_x = \langle   Q(x^{-1/2}) d_x, Q(x^{-1/2}) d_x \rangle.
    \]
    But $Q(x^{1/2}) = \sqrt\mu Q(w^{1/2})$. Hence, $Q(x^{-1/2}) d_x = d$, proving that
    $\|d\| = \|d_x\|_x$. That $\|d\| = \|d_s\|_s$ follows by similar argument.
      %\[ 
      %d_s = -\mu Q(x)^{-1} d_x =  -\mu Q(x)^{-1/2} d.
      %\]
      %But $\mu Q(x)^{-1/2}  =  \mu Q(x^{-1})^{1/2} = Q( \mu x^{-1})^{1/2} = Q( s)^{1/2}.
      %Hence
    \end{proof}
\end{prop}
These propositions suggest how {\tt shortstep}
generalizes to non-symmetric cones.
Indeed, any cone with a \emph{log-homogeneous, self-concordant barrier function}
is equipped with a natural Riemannian geometry~\cite{nesterov2002riemannian}
that enables definition of $\Exp_x$ and $\Exp_s$.
The affine constraints characterizing the tangent vectors (\Cref{prop:manifoldTangent}) 
also generalize if one interprets $Q(u)^{-1}$ 
as the Hessian of the barrier function $\log \det u^{-1}$.
One can also interpret $\mu Q(x^{-1})$ in \Cref{prop:manifoldTangent} as the \emph{parallel transport} 
operator from $x$ to $s$, a canonical operation in Riemannian geometry.
An obstruction to implementation, however,
is evaluation of $\Exp_x$ and $\Exp_s$, 
which, as mentioned, may require numerical solution of the ODE system~\eqref{eq:geodesic}.
Our convergence analysis (\Cref{sec:geodesic}) will also leverage
the spectral theory of symmetric cones, and hence does not immediately
generalize.

An alternative generalization, applicable to even symmetric cones, 
replaces $\Exp_u$ with a general \emph{retraction} $R_u : \algName \rightarrow \inter\coneName$.  Retractions
are defined by relaxing the geodesic property of $\Exp_u$. That is, a
retraction $R_u$ smoothly maps a tangent vector $v \in \algName$ to
a point $R_u(v) \in \inter \coneName$ with the property that
the curve $\gamma(t) := R_u(t v)$ satisfies $\gamma(0) = u$ and  $\dot \gamma(0) = v$.
The map $\Exp_u$ is a special case of a retraction for which $\gamma(t)$ is geodesic.
  See~\cite[Chapter 4.1]{absil2009optimization} for more details.

%\begin{lem}
%  The parallel transport from $x$ to $y$ on $\mathbb{S}^n$ is given by
%
%  \[
%    Q(x) Q(M(x^{-1}, y)) 
%  \]
%  where $M(x, y)$ is the geometric mean.
%
%  \begin{proof}
%    Proof: See appendix of Parallel Transport on the Cone Manifold of SPD Matrices for Domain Adaptation 
%and page 11 of https://arxiv.org/pdf/1312.1039.pdf
%  \end{proof}
%
%
%\end{lem}

\subsection{Newton direction via orthogonal projection}
We next derive an orthogonal, direct-sum decomposition of the Newton direction with respect to the subspaces 
$\mathcal{L}_w := \{ Q(\sqrtminvw) u :  u \in \mathcal{L} \}$
and  $\mathcal{L}_w^{\perp} = \{ Q(\sqrtmw) u : u \in \mathcal{L}^{\perp} \}$.
This decomposition establishes both its claimed uniqueness  (Definition~\ref{defn:NewtonDir})
and a formula for its construction via orthogonal projection.
\begin{prop}\label{prop:newProj}
  For $\mu > 0$ and $w \in \inter \coneName$, let
  \[
    d_1 =\proj_{\mathcal{L}^{\perp}_w }\bigg(Q(\sqrtminvw)(\frac{1}{\sqrt{\mu}}x_0 - w)\bigg), \qquad  d_2 =\proj_{\mathcal{L}_w }\bigg(Q(\sqrtmw)(\frac{1}{\sqrt{\mu}}s_0 - w^{-1})\bigg).
\]
  Then the Newton direction $\newton(w, \mu)$ satisfies $\newton(w, \mu) = d_1 - d_2$.
  \begin{proof}
    Let $r_1 = Q(\sqrtminvw)(\frac{1}{\sqrt{\mu}}x_0 - w)$ and $r_2 = Q(\sqrtmw)(\frac{1}{\sqrt{\mu}}s_0 - w^{-1})$.
    By the identity $Q(z^{1/2})e = z$ (\Cref{lem:quad}), the conditions of
    Definition~\ref{defn:NewtonDir} are equivalent to
    \[
      w +  Q(\sqrtmw) d \in \frac{1}{\sqrt \mu} x_0 + \mathcal{L}, \qquad w^{-1} -  Q(\sqrtminvw) d \in \frac{1}{\sqrt \mu} s_0 + \mathcal{L}^{\perp}.
    \]
    Using $Q(z^{-1}) = Q(z)^{-1}$ (\Cref{lem:quad}), we conclude that $d \in r_1  + \mathcal{L}_w$ and $d \in -r_2 + \mathcal{L}_w^{\perp}$.
    Equivalently,
    \[
      d \in  (\proj_{\mathcal{L}_w^{\perp}} (r_1)  + \mathcal{L}_w) \bigcap  (\proj_{\mathcal{L}_w} (-r_2)  + \mathcal{L}_w^{\perp}),
    \]
     since any affine set $z_0 + \mathcal{S}$ satisfies $z_0 + \mathcal{S} = \proj_{\mathcal{S}^{\perp}}(z_0) + \mathcal{S}$.
    Hence, $d$ has the following direct-sum decompositions with respect to $\mathcal{L}_w$ and $\mathcal{L}_w^{\perp}$:
    \[
      d = \proj_{\mathcal{L}_w^{\perp}} (r_1) + d_{\mathcal{L}_w}, \qquad d = \proj_{\mathcal{L}_w} (-r_2)  + d_{\mathcal{L}_w^{\perp}}.
    \]
    Since such decompositions are unique,  $d_{\mathcal{L}_w} = \proj_{\mathcal{L}_w} (-r_2)$, proving the claim.
  \end{proof}
\end{prop}
\noindent This decomposition has immediate practical implications: 
one can use any algorithm for orthogonal projection,
e.g., the Gram-Schmidt process or a least-squares method,  to find
$\newton$. Section~\ref{sec:leastsquares} gives an explicit
linear system for performing this projection using this latter approach.
Further, the size/structure of this linear system matches the size/structure
of linear systems arising in classical IPMs.
%illustrating that the cost of selecting $\newton$ at each iteration of 
%{\tt shortstep} matches the cost of selecting a search direction in a 
%typical IPM\@.
\subsection{Scale invariance}~\label{sub:affineinvar}
For an automorphism $T : \algName \rightarrow \algName$ of $\coneName$,
consider the transformed  primal-dual pair:
\begin{align} \label{sdp:tmain}
\begin{array}{ll}
  \mbox{ minimize } & \langle (T^{-1})^{*}  s_0, x \rangle     \\
  \mbox{ subject to } &  x \in \coneName \cap T( x_0 +   \mathcal{L})
\end{array}  \qquad
\begin{array}{ll}
  \mbox{ minimize } & \langle Tx_0, s \rangle     \\
  \mbox{ subject to } &  s \in \coneName \cap (T^{-1})^{*} ( s_0 + \mathcal{L}^{\perp}),
\end{array}
\end{align}
where $(T^{-1})^{*}: \algName \rightarrow \algName$ denotes the adjoint of $T^{-1} : \algName \rightarrow \algName$.  We next show the
following: if {\tt shortstep} maps input $w_0$ to output $\bar w$ for
the primal-dual pair~\eqref{sdp:main}, then it maps input $Tw_0$ to output  $T
\bar w$ for the transformed pair~\eqref{sdp:tmain}. In other words, it is
\emph{scale invariant} in the sense of~\cite{tunccel1998primal}. 
To show this, we first establish that the Newton direction $\hatnewton(w, \mu)$ 
for the transformed problem satisfies $\hatnewton(T w, \mu) = M \newton(w, \mu)$ 
for an automorphism $M$, dependent on $T$ and $w$, that is
also \emph{orthogonal}, i.e., $M^{-1} = M^*$.  Scale invariance will follow,
leveraging the fact that $\exp (M d) = M \exp(d)$ for any such
$M$~(\Cref{lem:auto}).

To give a formula for  $M$  and to establish its key properties, we use the
decomposition  $\newton(w, \mu) = d_1(w, \mu) - d_2(w, \mu)$ of the Newton
direction from~\Cref{prop:newProj}.  We also decompose the transformed
direction as $\hatnewton(v, \mu) = d_{1, T}(v, \mu)- d_{2, T}(v, \mu)$
by applying~\Cref{prop:newProj} to the transformed problem~\eqref{sdp:tmain}.
\begin{lem}~\label{lem:mapInvar}
  Let $M = Q(Tw)^{-1/2}T Q(w)^{1/2}$
  for $w \in \inter \coneName$
  and an automorphism $T: \algName \rightarrow \algName$ of $\coneName$.
  The following statements hold.
  \begin{enumerate}[label= (\alph*)]
    \item\label{lem:mapInvar:orth} $M$ is an orthogonal automorphism of $\coneName$.
    \item\label{lem:mapInvar:form} $M = Q(Tw)^{1/2}(T^{-1})^* Q(w)^{-1/2}$.
    \item\label{lem:mapInvar:mIdent} For all $\mu > 0$, the Newton directions 
    satisfy $\hatnewton(Tw, \mu) = M\newton(w, \mu)$.
      Further, their direct summands satisfy
      \begin{align*}
        d_{1, T}(Tw, \mu) = Md_1(w, \mu), \qquad d_{2, T}(Tw, \mu) = Md_2(w, \mu).
    \end{align*}
  \end{enumerate}
  \begin{proof}
    That $M$ is an automorphism follows because it is a composition of automorphisms.
   We next verify orthogonality, i.e., that $M^{-1} = M^*$:
    \begin{align*}
      M^* M &=   Q(w)^{1/2} T^* Q(Tw)^{-1}T Q(w)^{1/2} 
           =     Q(w)^{1/2} T^* (T Q(w) T^*)^{-1} T Q(w)^{1/2} 
            = I,
    \end{align*}
    where we've used the identities $Q(Tw) = TQ(w)T^*$ and $Q(w)^{1/2} Q(w)^{-1} Q(w)^{1/2} = I$  (\Cref{lem:quad}).
     Since by construction $M^{*} Q(Tw)^{1/2}(T^{-1})^* Q(w)^{-1/2} =  I$, orthogonality
     implies the next statement.
   
    By  definition of $M$ and the second property,  we conclude that $M
    Q(w)^{-1/2} = Q(Tw)^{-1/2}T$ and $M  Q(w)^{1/2} = Q(Tw)^{1/2}(T^{-1})^*$.
    Combining this with $Me = e$ (Lemma~\ref{lem:auto}) shows
    that both
    \[
      \hatnewton(Tw, \mu)  \in  M  \bigg(\frac{1}{\sqrtcp} Q(w^{-1/2}) x_0 -e + Q(w^{-1/2})\mathcal{L} \bigg) 
      \]
    and
    \[
 \hatnewton(Tw, \mu) \in M\bigg(e - Q(w^{ 1/2})\frac{1}{\sqrtcp}s_0 + Q(w^{ 1/2})\mathcal{L}^{\perp} \bigg).
    \]
    Following the proof of \Cref{prop:newProj}, we conclude
    that $\hatnewton(Tw, \mu)=  M (d_1 - d_2)$, which implies  
     $d_{1,T} = Md_1$ and $d_{2,T} = M d_2$. 
  \end{proof}
\end{lem}
We use this lemma to  show scale invariance of $w \leftarrow
Q(w^{1/2}) \exp( \alpha(d_1, d_2)  d)$, where $d = \newton$ and $\alpha :
\algName \times \algName \rightarrow \mathbb{R}$ is a step-size rule invariant
under transformation by $M$.
\begin{prop}~\label{prop:affine}
  Let $\alpha : \algName \times \algName \rightarrow \mathbb{R}$ be a function satisfying
  $\alpha(d_1, d_2) = \alpha(Md_1, Md_2)$ for any orthogonal automorphism $M : \algName \rightarrow \algName$.
  Then, for any automorphism $T: \algName \rightarrow \algName$,
  $w \in \inter \coneName$, and $\mu > 0$, 
  \[
    Q(\tilde w^{1/2} ) \exp \big( \alpha(\tilde d_1, \tilde d_2) \tilde d\big) = T  Q(w^{1/2}) \exp \big( \alpha(d_1, d_2)  d\big),
  \]
  where $\tilde w = Tw$, $d = \newton(w, \mu)$, $\tilde d = \hatnewton(\tilde w,
  \mu)$, $d_i = d_{i}(w, \mu)$, and $\tilde d_i = d_{i,T}(\tilde w, \mu)$ for $i\in\{1,2\}$.

  \begin{proof}
    Let $M = Q(Tw)^{-1/2}T Q(w)^{1/2}$.
     By Lemma~\ref{lem:mapInvar}, $M$ is an orthogonal automorphism. Hence,
     $\exp (M x) = M \exp (x)$ for all $x$ (\Cref{lem:auto}). Combining this with \Cref{lem:mapInvar}\ref{lem:mapInvar:mIdent} yields
  \begin{align*}
    Q(\tilde w^{1/2}) \exp  (\alpha(\tilde d_1,\tilde d_2 ) \tilde d) &=   Q(\tilde w^{1/2}) \exp (\alpha(M d_1, M d_2) M d)  \\
                                                              &= Q(\tilde w^{1/2})  M \exp  (\alpha( M d_1, M d_2) d).
  \end{align*}
    But $\alpha(M d_1, M d_2) = \alpha(d_1, d_2)$ by assumption and $Q(\tilde w^{1/2})  M = TQ( w^{1/2})$ by definition of $M$ 
    and the identity $Q(u)^{1/2} = Q( u^{1/2})$; see \Cref{lem:quad}.
  \end{proof}
\end{prop}
\noindent Scale invariance of {\tt shortstep} follows by invoking this result 
at each iteration with the step-size $\alpha(d_1, d_2) = 1$. We will use a nontrivial step-size rule in Section~\ref{sec:prac}.

\subsection{Comparison with the Nesterov-Todd algorithm}\label{sec:ntcomp}
The celebrated algorithm of Nesterov~and~Todd (NT)~\cite[Section 6]{nesterov1998primal},
which extends the linear programming algorithms of~\citet{kojima1989primal}~and~\citet{monteiro1989interior},
shares key properties with {\tt shortstep}: it
is scale invariant, it executes $\bigO(\sqrt{n})$ iterations, it is primal-dual symmetric,
and finding its search direction reduces to orthogonal projection. 
This suggests a fundamental connection with {\tt shortstep}.
In general, iterations of the NT algorithm do \emph{not} satisfy $x = \mu s^{-1}$.
However, if this relation holds, then the NT search direction coincides with our
Newton direction.  Further, its $(x, s)$-update is a first-order approximation
of our geodesic update.

To see this, note that the NT direction is, in the framework of Jordan algebras~\cite[Section 3.2]{gu2011full},
 the unique $(d_x, d_s) \in \algName \times \algName$ satisfying
\begin{equation}~\label{eq:nesteq}
  x + \sqrt{\mu} Q(p^{1/2}) d_x \in x_0 + \mathcal{L}, \;\; s + \sqrt{\mu} Q(p^{-1/2}) d_s \in s_0 + \mathcal{L}^{\perp}, \;\; d_x +  d_s = v^{-1} - v,
\end{equation}
where $p$ is the \emph{scaling point}, defined as $Q(x^{1/2})(Q(x^{1/2})s)^{-1/2}$,
and $v := \frac{1}{\sqrt{\mu}}Q(p^{-1/2}) x$. 
Given $(d_x, d_s)$, the NT algorithm updates  $(x,s)$ to  $(x', s')$, where
\begin{align}~\label{eq:nestupdate}
x' := x + \sqrt{\mu}Q(p^{1/2}) d_x, \qquad s' := s  + \sqrt{\mu} Q(p^{-1/2})  d_s.
\end{align}
Our result follows.
%Further, 
%$(x', s')$ can be viewed as first-order approximations of the geodesic updates 
%\[
%x = \sqrt{\mu}Q(w^{1/2})\exp d,  \qquad s = \sqrt{\mu}Q(w^{-1/2})\exp(-d),
%\]
%from \Cref{prop:paramGeo}.
\begin{prop}
  Let $x, s \in \inter \coneName$  satisfy 
  $x = \mu s^{-1}$ for  $\mu > 0$. Let $w = \frac{1}{\sqrt \mu} x$
  and $d = \newton(w, \mu)$. Then,
  \begin{enumerate}[label= (\alph*)]
    \item $p = w$, where $p$ is the scaling point $Q(x^{1/2})(Q(x^{1/2})s)^{-1/2}$.
    \item  $d_x = d$ and $d_s = -d$ where $(d_x, d_s)$ is the NT direction~\eqref{eq:nesteq}.
    \item  $x' =\sqrt{\mu} Q(w^{1/2})(e  + d)$ and  $s' =  \sqrt{\mu}Q(w^{-1/2})(e -  d)$, where
      $(x', s')$ is the NT update~\eqref{eq:nestupdate} and $e+d$ and $e-d$ are the
      first-order Taylor-expansions of $\exp(d)$ and $\exp(-d)$ at $d = 0$.
  \end{enumerate}
  \begin{proof}
    If $x = \mu s^{-1}$, then the definitions of $w$ and the scaling point $p$ easily
    imply that $p = w$  and $v = e$.  We also conclude that $d_x + d_s = v^{-1} - v =
    0$.  Combining these identities with  $w=Q(w^{1/2}) e$, 
    $w^{-1}=Q(w^{-1/2})e$, and~\eqref{eq:nesteq} yields
\begin{align*}
  \sqrt{\mu}Q(w^{1/2})(e +  d_x) \in x_0 + \mathcal{L}, \qquad \sqrt{\mu}Q(w^{-1/2})(e -  d_x) \in s_0 + \mathcal{L}^{\perp}, 
\end{align*}
    which are the defining conditions of $\newton(w, \mu)$ given by \Cref{defn:NewtonDir}. Hence, $d = d_x$.
    Finally, the claimed formula for $(x', s')$ holds because $x = \sqrt{\mu} Q(w^{1/2})e$ and $s = \sqrt{\mu} Q(w^{-1/2})e$.
  \end{proof}
\end{prop}

Note with the stronger assumption that $x = s^{-1}$, we can similarly interpret
algorithms based on the so-called H..K..M~direction since, in this case, it
coincides with the NT direction~\cite{todd1998nesterov}.  It was introduced
independently by~\citet{helmberg1996interior},~\citet{kojima1997interior}~and~\citet{monteiro1997primal}.
Also note that even if $x = \mu s^{-1}$ fails, the scaling point $p$ still has
a Riemannian interpretation: it is precisely the midpoint of the geodesic
connecting $x$ and $s^{-1}$, or, equivalently, their geometric mean~\cite{lim2000geometric}.
Finally, we note that the NT direction has an alternative
derivation due to \citet{sturm1999symmetric}; see remarks in~\cite{sturm2002implementation}.

\section{Geodesics and divergence}~\label{sec:geodesic}
The  goal of this section is to prove the $\mu$-update and centering lemmas
used in the analysis of {\tt shortstep} (\Cref{alg:barrier}).  Towards this, we
first study a proxy for geodesic distance $\delta(u, v)$  that is easier to
bound during the course of Newton's method.  This proxy generalizes  the
\emph{symmetric Kullback-Leibler divergence} $h(U, V):=\trace(U V^{-1}
+ U^{-1} V - 2I)$ of two zero-mean Gaussian distributions with covariance 
matrices $U$ and $V$, also known as the
\emph{Jeffrey divergence}~\cite{harandi2017dimensionality,
moakher2006symmetric}.  We hence call this proxy \emph{divergence}. 
We define it using the fact that $\tr e$ equals the rank of $\coneName$ (which we've denoted by $n$).
\begin{defn}\label{defn:merit}
  Denote by
  $h(u, v)$ the \emph{divergence} of $u, v \in \inter \coneName$, defined as $h(u, v)  :=   \langle u, v^{-1}\rangle +  \langle u^{-1},  v \rangle - 2n$.
\end{defn}
\noindent Divergence is symmetric and non-negative, i.e., $h(u, v) = h(v, u)$
and $h(u, v) \ge 0$ for all $u, v \in \inter\coneName$. Further, $h(u, v) = 0$
if and only if $u = v$.  However, unlike geodesic distance $\delta(u, v)$, it
is \emph{not} a metric, as  the triangle inequality can fail.

Recall from Lemma~\ref{lem:GeoProp} that geodesic distance satisfies $\delta(u,
v) = \|\log {Q(v^{-1/2})} u\|$.  Equivalently,  $\delta(u, v)^2 = \sum_{\lambda
\in S} \lambda^2$, where $S$ denotes the multiset of eigenvalues of $\log
{Q(v^{-1/2})} u$.  This formula holds for divergence 
if we replace  $\lambda^2$ with the 
upper bound $q(\lambda) :=  2(\cosh(\lambda) -1)$ introduced in \Cref{sec:alg}.
\begin{lem}~\label{lem:divSpec}
For all $u, v \in \inter \coneName$, the divergence satisfies $h(u, v) =
\sum_{\lambda \in S} \qup(\lambda)$, where $S$ is the multiset of eigenvalues
of $\log {Q(v ^{-1/2})} u$. 
\end{lem}
\noindent This enables us to prove the following bounds relating divergence
to geodesic distance.
\begin{lem}~\label{lem:divdistBound}
Let $u, v \in \inter \coneName$. Then,  ${\delta(u, v)}^2 \le  h(u, v) \le \qup(\delta(u, v))$.
  \begin{proof}
    Let $\lambda \in \mathbb{R}^n$ denote the vector of eigenvalues of $\log {Q(v^{-1/2})} u$.
    The lower bound follows from~\Cref{lem:divSpec} and \Cref{lem:GeoProp}\ref{lem:GeoProp:normal}
    given that $q(\lambda_i) \ge \lambda_i^2$.
    To prove the upper bound, it suffices to show that
    $\sum^n_{i=1} (\cosh(\lambda_i) - 1) \le \cosh(\|\lambda\|) - 1$.
   To begin, consider the upper bound
    \[
      \sum^n_{i=1} (\cosh(\lambda_i) - 1) \le \sup_{\|z\|=\|\lambda\|} \sum^n_{i=1}(\cosh(z_i) - 1).
    \]
    Let $z$ achieve the supremum.  Then it must be a critical point,
    which implies existence of $\gamma \in \mathbb{R}$ satisfying 
   $\gamma z + \sinh(z) = 0$. We conclude that $z_i = 0$ or $|z_i| = c$ for a  constant $c > 0$.
   We now claim that $z_i \ne 0$ and $z_j \ne 0$ implies $i = j$. Suppose otherwise.
   Then we don't change $\|z\|$ by setting
   $z_i = 0$ and $z_j = \sqrt{2} c$. Further, we increase
    $\sum^n_{i=1}(\cosh(z_i) - 1)$ given that $\cosh( \sqrt{2}c   ) - 1 >  2 (\cosh(c) - 1)$,
    contradicting our assumption that $z$ attains the supremum.
  \end{proof}
\end{lem}

  We also note that $h(u, v)$ shares
  the invariance properties of geodesic distance $\delta(u, v)$.
  It is symmetric with respect to inversion, i.e., $h(u, v) = h(u^{-1}, v^{-1})$.
  Hence, it measures the proximity of  $(w,  w^{-1})$ 
  to the centered-points $ (\hat w(\mu),   \hat w(\mu)^{-1})$ in a primal-dual
  symmetric way, i.e., $h( w, \hat w(\mu)) = h( w^{-1}, \hat w(\mu)^{-1})$.
  It is also scale invariant, meaning $h(T u, T v) = h(u, v)$ for any
  automorphism $T$ of $\coneName$.

\begin{rem}
The quantity $h(v, v^{-1})$ where $v$ is as defined in Section~\ref{sec:ntcomp},
is used to analyze a full-step Nesterov-Todd algorithm~\cite[Section 3.3]{gu2011full}.
  %Also note that for fixed $z_1$, the divergence  $h(z_0, z_1)$ consists of
  %a term $\langle z_0, z_1^{-1}\rangle$ that is linear in $z_0$ and another
  %term $\langle z^{-1}_0, z_1\rangle$ that is linear in $z^{-1}_0$.
  %These symmetry and linearity properties will be exploited to prove
  %the $\mu$-update and centering lemmas.
\end{rem}

\subsection{Divergence along the central path}
Divergence has the following utility: we can  calculate it \emph{exactly}
for two centered points $\hat w(\cp_0)$ and $\hat w(\cp_1)$  even if we do 
not know these points explicitly.  Instead, all we need is the ratio of the centering parameters
$\cp_0$ and $\cp_1$ and the rank of $\coneName$, denoted by $n$.
\begin{thm}~\label{thm:centeringUpdateDiv}
Let $\mu_0, \mu_1 > 0$. Then,
  $\frac{1}{n} h(\hat w(\mu_0), \hat w(\mu_1)) = \qup(\frac{1}{2} \log \frac{\mu_0}{\mu_1})$.
  \begin{proof}
    Let   $u = \hat w(\mu_0)$, $v = \hat w(\mu_1)$ and  $\alpha = \sqrt{\frac{\mu_0}{\mu_1}}$.  Since 
   $\sqrt\mu_0(u, u^{-1})$ and $\sqrt\mu_1(v, v^{-1})$ are feasible,
  \[
    v - \alpha u \in \mathcal{L}, \qquad  v^{-1} - \alpha {u}^{-1} \in \mathcal{L}^{\perp}.
  \]
  Hence,
    $0 = \langle v -  \alpha u, v^{-1} - \alpha {u}^{-1} \rangle  = (1+\alpha^2 )n  - \alpha\langle  v, {u}^{-1} \rangle  -  \alpha\langle u, v^{-1} \rangle$.
    Rearranging shows that
    \[
      \langle  u,  v^{-1} \rangle  +   \langle   {u}^{-1}, v \rangle = n \frac{1+\alpha^2}{\alpha} = n (\alpha + \frac{1}{\alpha}) = 2n(\cosh(\log(\alpha))).
    \]
    Hence, $h(u, v) =  2n(\cosh (\log(\alpha)) - 1)$. 
    Using  $q(t) := 2(\cosh(t) - 1)$ and
    $\log \alpha =\frac{1}{2} \log \frac{\mu_0}{\mu_1}$
    yields:
  \[
 \frac{1}{n} h(u, v) = q(\log \alpha) = q(\frac{1}{2} \log \frac{\mu_0}{\mu_1}).
  \]
  \end{proof}
\end{thm}
\noindent Combining this theorem with the bounds relating divergence
and geodesic distance (\Cref{lem:divdistBound}) lets us prove the $\mu$-update lemma,
which we reproduce below.
\centeringUpdate*
\begin{proof}
  From \Cref{thm:centeringUpdateDiv}, we conclude that
    $\frac{1}{n}h(\hat w(\mu),  \hat w(\frac{1}{k} \mu )) =  q(\frac{1}{2} \log k)$.
  Since $\delta(\hat w(\mu),  \hat w(\frac{1}{k} \mu ))^2 \le h(\hat w(\mu),  \hat w(\frac{1}{k} \mu ))$ by  \Cref{lem:divdistBound}, the claim follows.
\end{proof}
\begin{rem}
Since geodesic distance is invariant under inversion and positive
  rescaling,
 we have, for $(x, s) = \sqrt\mu(w, w^{-1})$, that $\delta(x, \hat x (\mu)) = \delta(s, \hat s (\mu)) = \delta(w, \hat w (\mu))$. This implies that
the lengths $L_x$ and $L_s$ of the primal and dual central paths
also upper bound $\delta(\hat w(\mu_0), \hat w(\mu_1))$, where
\[
  L_x := \int^{\mu_1}_{\mu_0} \| \frac{d}{d \mu} \hat x(\mu)\|_{\hat x(\mu)}   d \mu, \qquad  L_s := \int^{\mu_1}_{\mu_0} \| \frac{d}{d \mu} \hat s(\mu)\|_{\hat s(\mu)}  d \mu,
\]
and $\|v\|_u :=  \| Q(u)^{-1/2 }v\|$.  
Bounds on $L_x$  in terms of $\log(\mu_0/\mu_1)$ \emph{and}
  the (generally unknown) values of the barrier function $\log \det z^{-1}$ 
 at $z=\hat x(\mu_0)$ and $z=\hat x(\mu_1)$ appear in~\cite[Lemma 4.1]{nesterov2008primal}.
\end{rem}

\subsection{Divergence along geodesics}
Fix $\mu > 0$, $w \in \inter \coneName$, and nonzero $d \in \algName$, and define the function
$f : \mathbb{R} \rightarrow \mathbb{R}$
\[
  f(t) =  h\left(Q(w^{1/2}) \exp(t d), \hat w (\mu) \right).
\]
That is, let $f(t)$ return the divergence between the centered point $\hat w(\cp)$  and points
on the geodesic induced by $(w, d)$.  Though we don't know $\hat w(\cp)$
and hence cannot evaluate $f$, we can still establish crucial properties, such as its strict convexity.
\begin{lem}~\label{lem:strictConvex}
  The function $f$ is strictly convex.
  \begin{proof}
    Let $a:=Q(w^{1/2}) \hat w (\mu)^{-1}$ and let  $\sum^n_{i=1} \lambda_i e_i$
    denote the spectral decomposition of $d$.  Then,
\[
  f(t) + 2n = \langle a, \exp(td)\rangle + \langle a^{-1}, \exp(-td)\rangle =   \sum^n_{i=1} \exp(t\lambda_i) \langle  a, e_i \rangle  + \exp(-t\lambda_i) \langle a^{-1}, e_i  \rangle.
\]
But $\langle  a, e_i \rangle > 0$ and $\langle a^{-1}, e_i  \rangle >0$ since
    $a, a^{-1} \in \inter \coneName$ and $e_i \in \coneName$, proving the claim
    by strict convexity of the scalar exponential function.
  \end{proof}
\end{lem}

\noindent 
We can also ensure that $f(t) < f(0)$ for a piecewise step-size rule
involving the \emph{spectral norm} $\|d\|_{\infty}$ of $d$, 
defined as $\|d\|_{\infty} := \max_{\lambda \in S} |\lambda|$
where $S$ denotes the multiset of eigenvalues of $d$.
This rule also incorporates a parameter  $\theta \in (0, 1)$
controlling the transition from full to damped Newton steps. 
\begin{thm}~\label{thm:newton}
  Let $d = \newton(w, \mu)$ and $\alpha=\max\{1, \frac{1}{2\theta}   \|d\|^2_{\infty} \}$
  for $\theta \in (0, 1)$.  The following statements hold. 
  \begin{enumerate}[label= (\alph*)]
    \item~\label{item:thm:newton:full}~If  $\alpha=1$, then $f(1) \le \frac{1}{2} \|d\|^2_{\infty} f(0) \le \theta f(0)$.
    \item~\label{item:thm:newton:damped} $f(1/\alpha) <  f(0)$.
      %+ (t^2-t) \|d\|^2 \le f(0)$.
  \end{enumerate} 
\end{thm}

\noindent To prove this theorem, we'll first provide the derivatives of $f(t)$ and
a descent condition on $t$ for arbitrary $d$.  We then specialize  results
to the Newton direction $\newton(w, \mu)$. 

\begin{rem}
  A function $p : \inter \coneName \rightarrow \mathbb{R}$ is called \emph{geodesically convex} if its
  restrictions to  geodesics are convex in the usual sense, i.e., if $p(
  g(t))$ is a convex function of $t$ for all curves $g(t)$ of the
  form $t \mapsto Q(w^{1/2}) \exp(td)$. 
  This convexity notion is a central concept in manifold optimization~\cite{sra2015conic, wiesel2012geodesic, duembgen2016geodesic}.
  The convexity of $f(t)$ reflects
  the geodesic convexity of the divergence map $w \mapsto h(w, \hat w(\mu))$.
\end{rem}

\subsubsection{Derivatives and descent condition}\label{sec:descent}
\noindent The derivatives $d^m f(t)/(dt)^m$, denoted $f^{(m)}$ for short,  have a concise
form given the role of the exponential function in the definition of $f$.  Interpreting $f(t)$
as the trace of a particular point in $\coneName$ also allows
us to bound \emph{even} derivatives using just $d$ and $f(t)$.
\begin{lem}~\label{lem:costlongt}
  Let $a(t) =  Q(\exp(td))^{1/2}   Q(\sqrtmw) \hat w(\mu)^{-1}$.  The following hold for all $t \in \mathbb{R}$:
  \begin{enumerate}[label= (\alph*)]
    \item\label{lem:costlongt:f} $f(t) = \tr (a(t) + a(t)^{-1} - 2e)$, where  $a(t) + a(t)^{-1} - 2e \in \coneName$.
    \item\label{lem:costlongt:df} $f^{(m)}(t) =   \langle a(t) +  {(-1)}^{m} a(t)^{-1} ,    d^m    \rangle$.
    \item\label{lem:costlongt:dfeven} $f^{(2m)}(t) \le \|d\|^{2m}_{\infty} f(t) + 2 \langle e, d^{2m} \rangle$.
   \end{enumerate}
  \begin{proof}
    By definition of $f$ and divergence (Definition~\ref{defn:merit}), we have
    \[
      f(t) = \langle  Q(\sqrtmw)  \exp(t d)  ,  \hat w(\mu)^{-1} \rangle + 
              \langle  Q(\sqrtminvw)  \exp(-t d)  ,  {\hat w(\mu)} \rangle - 2n.
    \]
    Substituting $\exp(t d)= Q(\exp(td))^{1/2} e$ and $\exp(-t d)= Q(\exp(-td))^{1/2} e$
    yields
    \[
      f(t) = \langle    e  ,   Q(\exp(td))^{1/2}  Q(\sqrtmw)\hat w(\mu)^{-1} \rangle + 
              \langle e , Q(\exp(-td))^{1/2}  Q(\sqrtminvw)   {\hat w(\mu)} \rangle - 2n.
    \]
    Two applications of the identity $[Q(u)v]^{-1} = Q(u^{-1}) v^{-1}$ 
    shows that
    \[
    a(t)^{-1} = Q(\exp(-td))^{1/2}   Q(\sqrtminvw) \hat w(\mu),
    \]
    which proves the trace identity of the first statement.
    That $a(t) + a(t)^{-1} - 2e \in \coneName$
    follows because each eigenvalue has form $\lambda + \frac{1}{\lambda} - 2$ for some $\lambda > 0$, which is always
    nonnegative.
    
    For~\ref{lem:costlongt:df},  we have that $\frac{d^m}{{dt}^m}
    \exp(t d)= d^m \circ \exp(td)= Q(\exp(td))^{1/2} d^m$. 
    This implies that
  \begin{align*}
    \frac{d^m}{dt^m}  \langle e, a(t) \rangle
         &=  \langle Q(\sqrtmw) \hat w(\cp)^{-1} ,\frac{d^m}{{dt}^m} \exp(t d)   \rangle \\
         &=  \langle Q(\exp(td))^{1/2} Q(\sqrtmw) \hat w(\cp)^{-1} ,    d^m    \rangle \\
         &=  \langle a(t) ,    d^m    \rangle.
  \end{align*}
    By similar argument, $\frac{d^m}{dt^m}  \langle e, a(t)^{-1} \rangle = ({-1})^m\langle a(t)^{-1} ,    d^m    \rangle$.
We conclude for all integers $m \ge 1$ that $f^{(m)}(t) =   \langle a(t) +  {(-1)}^{m} a(t)^{-1} ,    d^m    \rangle$.
For statement~\ref{lem:costlongt:dfeven},
  we have, since $a(t)  + a^{-1}(t) - 2e \in \coneName$, 
  \begin{align*}
    f^{(2m)}(t) &= \langle  a(t)  + a^{-1}(t) - 2e, d^{2m} \rangle   + 2 \langle e, d^{2m} \rangle \\
                &\le \|a(t)  + a^{-1}(t) - 2e\|_1 \|d\|^{2m}_{\infty}  + 2 \langle e, d^{2m} \rangle \\
                &= \tr (a(t)  + a^{-1}(t) - 2e) \|d\|^{2m}_{\infty}  + 2 \langle e, d^{2m} \rangle \\
                &= \|d\|^{2m}_{\infty} f(t)   + 2 \langle e, d^{2m} \rangle.
  \end{align*}

  \end{proof}
\end{lem}

Let $f'$ and $f''$ denote the first and second derivatives of $f$.
Assuming $f'(0) < 0$, we now  provide a descent condition on $t$, i.e., 
we establish an interval on which  $f(t) \le f(0)$. Our analysis
rests on Taylor's theorem, convexity of $f$,  and the \Cref{lem:costlongt}  bound on $f''(t)$.

\begin{lem}~\label{lem:descent}
  If $f'(0) < 0$ and $0 \le t \le  \frac{-2f'(0)}{  \|d\|^2_{\infty} f(0) +2  \|d\|^2}$, then $f(t) \le f(0)$.
  \begin{proof}

  By Taylor's theorem, $f(t) = f(0) + f'(0) t + \frac{1}{2} f''(\zeta) t^2$ for
  some $\zeta \in [0, t]$.  Further, 
  \begin{align*}
    f''(\zeta) \le \|d\|^2_{\infty} f(\zeta)   + 2 \|d\|^2  \le  \max_{u \in \{0, t\}} \|d\|^2_{\infty} f(u)   + 2 \|d\|^2,
  \end{align*}
  where the first inequality is \Cref{lem:costlongt}\ref{lem:costlongt:dfeven}
    and the second inequality uses convexity of $f(t)$. Hence,
  \begin{align}~\label{eq:quadbound}
    f(t) \le f(0)  + f'(0) t + \frac{1}{2} \max_{u \in \{0, t\} } (\|d\|^2_{\infty} f(u) +2  \|d\|^2) t^2.
  \end{align}
  Now, let $\hat t$ be the smallest $t > 0$ for which $f(\hat t) = f(0)$. Then
  \[
   f(0) \le f(0)  + \hat t f'(0) + \frac{1}{2}  (\|d\|^2_{\infty} f(0) +2  \|d\|^2) \hat t^2,
  \]
  which implies that
  \[
    \hat t \ge \frac{-2 f'(0)}{\|d\|^2_{\infty} f(0) +2  \|d\|^2}.
  \]
  Since $f(t) \le f(0)$ for all $0 \le t \le \hat t$, the claim follows.
  \end{proof}
\end{lem}

\subsubsection{Newton direction}
Suppose now that $d = \newton(w, \mu)$. For this direction, we can bound $f'(0)$ 
using $f(0)$ and $\|d\|^2$ by applying the orthogonal, direct-sum decomposition of $d$ 
from \Cref{prop:newProj}. Recall that this decomposition is with respect to  $\mathcal{L}_w := \{ Q(\sqrtminvw) u :  u \in \mathcal{L} \}$
and  $\mathcal{L}_w^{\perp} = \{ Q(\sqrtmw) u : u \in \mathcal{L}^{\perp} \}$. This bound
also provides an updated descent condition for $t$.
\begin{lem}~\label{lem:Newtondesc}
  Suppose that $d = \newton(w, \mu)$. Then $f'(0)  = -(f(0) +  \|d\|^2)$. Further, $f(t) \le f(0)$ if
  \[
0\le    t \le \frac{2(f(0) + \|d\|^2)}{\|d\|^2_{\infty} f(0) +2  \|d\|^2}.
  \]
  \begin{proof}
    Let $r_1(t) = a(t)^{-1} -e$ and $r_2(t) = a(t) - e$, where $a(t)$ is as in
    Lemma~\ref{lem:costlongt}. Then, by
    \Cref{lem:costlongt}\ref{lem:costlongt:df},
    \begin{align*}
      -f' &= \langle  a^{-1} - a, d   \rangle = \langle  a^{-1} -e + e - a, d \rangle = \langle  r_1 - r_2, d \rangle.  
    \end{align*}
    Setting $t=0$ and substituting $d = \proj_{\mathcal{L}^{\perp}_w}r_1(0) -
    \proj_{\mathcal{L}_w} r_2(0)$  using \Cref{prop:newProj} gives
    \begin{align*}
      -f'(0)    = \langle  r_1 - r_2, d \rangle 
                = -\langle r_1, r_2 \rangle +  \|\proj_{\mathcal{L}^{\perp}_w} r_1\|^2 + \|\proj_{\mathcal{L}_w} r_2\|^2 
                = -\langle r_1, r_2 \rangle  +  \|d\|^2.
    \end{align*}
    But $f(t) = -\langle r_1(t), r_2(t) \rangle $ by \Cref{lem:costlongt}\ref{lem:costlongt:f}, proving the first claim.
The descent condition (\Cref{lem:descent}) specialized to the Newton direction
$d = \newton$ proves the second claim.
  \end{proof}
\end{lem}
\subsubsection{Proof of~\Cref{thm:newton}}
We can now prove the properties of the step-size rule $t=\min\{1, \frac{2\theta}{\|d\|^2_{\infty}}\}$
claimed by Theorem~\ref{thm:newton} for the Newton direction $d=\newton(w, \mu)$
and parameter $\theta \in (0, 1)$.
Assume first that $t = 1$.  Then $\|d\|^2_{\infty}  \le 2\theta < 2$, which, by \Cref{lem:Newtondesc},
implies that $f(1) \le f(0)$. Combining this with the 
quadratic upper bound~\eqref{eq:quadbound} and
$f'(0)  = -(f(0) +  \|d\|^2)$ from~\Cref{lem:Newtondesc} yields
    \[
  f(1) \le f(0) - (f(0) + \|d\|^2) + \frac{1}{2} (\|d\|^2_{\infty} f(0) +2  \|d\|^2)
        =    \frac{1}{2} \|d\|^2_{\infty} f(0) \le \theta f(0),
    \]
which is precisely the claim of~\Cref{thm:newton}-\ref{item:thm:newton:full}.
    Now suppose that $t = \frac{2\theta}{\|d\|^2_{\infty}} \le 1$. 
    By \Cref{lem:Newtondesc} and strict convexity of $f$, we have $f(t) < f(0)$ if
    \begin{align*}
      \frac{2\theta}{\|d\|^2_{\infty}}  ( \|d\|^2_{\infty}f(0) + 2 \|d\|^2) < 2  f(0) + 2\|d\|^2.
    \end{align*}
    But this inequality follows since $0 < \theta < 1$ and $\frac{2\theta}{\|d\|^2_{\infty}} \le 1$.
    Hence, $f(t) < f(0)$,  which is the claim of~\Cref{thm:newton}-\ref{item:thm:newton:damped}.
    %which, under our assumptions, hold if and only if $f(0) + 2 \frac{\|d\|^2}{\|d\|^2_{\infty}} \le  f(0) + \|d\|^2$.
    %We conclude that $f(t) \le f(0)$ since $\|d\|^2_{\infty} \ge 2$. By~\eqref{eq:quadbound},
    %it further holds that
    %\begin{align*}
    %  f(t) &\le f(0) - \frac{2}{\|d\|^2_{\infty}} (f(0) + \|d\|^2) + \frac{1}{2}(\|d\|^2_{\infty} f(0) + 2 \|d\|^2) \frac{4}{\|d\|^4_{\infty}} \\
    %        &= f(0) - \frac{2}{\|d\|^2_{\infty}} \|d\|^2 +   \|d\|^2 \frac{4}{\|d\|^4_{\infty}},
    %\end{align*}
%  where the last inequality follows given that $(\frac{1}{\|d\|^4_{\infty}}  -    \frac{1}{\|d\|^2_{\infty}}) < 0$ since $\frac{2}{\|d\|^2_{\infty}} \le 1$.
%  Hence,
%  \[
%    f(t) \le        f(0) + 2\|d\|_{\infty}^2 (\frac{1}{\|d\|^4_{\infty}}  -    \frac{1}{\|d\|^2_{\infty}}) \le f(0) +  2(\frac{1}{2} - 1), \\
%  \]

\subsection{Divergence bounds}
Though the centered point $\hat w(\cp)$ is unknown, the Newton direction
$\newton(w, \cp)$ can provide a lower bound $h_{lb}$ of the divergence $h(w,
\hat w(\cp))$ for any $w \in \inter \coneName$ and $\mu > 0$.  Under a norm
condition, we can also obtain an upper bound $h_{ub}$  and relative-error estimates; precisely, we can obtain $h_{ub}$ and
$ \alpha \ge 1$ satisfying
\begin{align}~\label{eq:est}
  h(w, \hat w(\mu)) \ge h_{lb} \ge  \frac{1}{\alpha}  h(w, \hat w(\mu)) \qquad  
  h(w, \hat w(\mu)) \le h_{ub} \le  \alpha  h(w, \hat w(\mu)).
\end{align}
These bounds use the direct-sum decomposition
$\newton = d_1- d_2$ from \Cref{prop:newProj} induced by the subspaces
$\mathcal{L}_w := \{ Q(\sqrtminvw) u :  u \in \mathcal{L} \}$ and
$\mathcal{L}_w^{\perp} = \{ Q(\sqrtmw) u : u \in \mathcal{L}^{\perp} \}$.
\begin{thm}~\label{thm:fdinfbound}
For $\cp > 0$ and $w \in \inter \coneName$, let $d = \newton(w, \mu)$,  $d_1 =
  \proj_{\mathcal{L}^{\perp}_w} d$, and $d_2 = -\proj_{\mathcal{L}_w} d$.  The
  following statements hold:
  \begin{enumerate}[label= (\alph*)]
    \item $h(w, \hat w(\mu)) \ge h_{lb}$ for $h_{lb} := \frac{\|d\|^2}{1+\|d_1 + d_2\|_{\infty}}$. 
    \item If $\|d_1 + d_2\|_{\infty} < 1$, then $h(w, \hat w(\mu)) \le h_{ub}$ for $h_{ub} := \frac{\|d\|^2}{1-\|d_1 + d_2\|_{\infty}}$.
      Further, the relative-error estimates~\eqref{eq:est} hold for 
      $\alpha =\frac{1+ \|d_1 + d_2\|_{\infty}}{1- \|d_1 + d_2\|_{\infty}}$.
  \end{enumerate}
  \begin{proof}
    Let $a =  Q(w^{1/2}) \hat w(\mu)^{-1}$,  $z = a + a^{-1} - 2e$ and $g =a - a^{-1}$. 
    \Cref{prop:newProj} implies that
      $d_1 = \proj_{\mathcal{L}_w^{\perp}}(a^{-1} - e)$ and $d_2 = \proj_{\mathcal{L}_{w}}(a-e)$.
      From $d=d_1-d_2$, we conclude
    \begin{align*}
      \proj_{\mathcal{L}_w^{\perp}}(g+2d)   =  \proj_{\mathcal{L}_w^{\perp}}(a - a^{-1} + 2(a^{-1} -e)   )  =  \proj_{\mathcal{L}_w^{\perp}}(a + a^{-1}  -2 e) = \proj_{\mathcal{L}_w^{\perp}}z,
    \end{align*}
    and, similarly,  that $\proj_{\mathcal{L}_w}(g+2d) = - \proj_{\mathcal{L}_w} z$.
    This implies that $\langle g+2d, d\rangle = \langle z, d_1  + d_2 \rangle$. Hence,
    \[
  -\|z\|_1 \|d_1  + d_2\|_{\infty}   \le -\langle g+2d, d\rangle \le \|z\|_1 \|d_1  + d_2\|_{\infty}.
    \]
    But from \Cref{lem:Newtondesc}, we also have that $-\langle g+2d, d\rangle
    = h(w, \hat w(\mu)) - \|d\|^2$.  Hence,
    \[
      -\|z\|_1 \|d_1  + d_2\|_{\infty}  \le h(w, \hat w(\mu))  - \|d\|^2 \le  \|z\|_1 \|d_1  + d_2\|_{\infty}.
    \]
    Using the fact that $\|z\|_1 =h(w, \hat w(\mu)) $ from
    \Cref{lem:costlongt}\ref{lem:costlongt:f} and rearranging these
    inequalities gives
    \[
      h(w, \hat w(\mu))  (1+\|d_1  + d_2\|_{\infty})  \ge  \|d\|^2 \ge  h(w, \hat w(\mu))  (1-\|d_1  + d_2\|_{\infty} ).
    \]
    Dividing by $1+\|d_1  + d_2\|_{\infty}$  proves the formula and error estimate for $h_{lb}$.
    Dividing by $1-\|d_1  + d_2\|_{\infty}$ proves the same for $h_{ub}$.
  \end{proof}
\end{thm}
Observe that we also obtain valid bounds by replacing $\|d_1  + d_2\|_{\infty}$ with $\|\newton(w, \mu)\|$
given that
$\|d_1 + d_2\|_{\infty} \le \|d_1 + d_2\|  = \|d_1 - d_2\| = \|\newton(w, \mu)\|$.
This in turn allows us to bound the size of Newton steps assuming bounds on divergence.
\begin{cor}~\label{cor:stepboundfromDiv}
For $\cp > 0$ and $w \in \inter \coneName$, suppose that
   $h(w, \hat w(\cp)) \le \frac{1}{2}$. Then, $\|\newton(w, \cp)\| \le 1$.
  \begin{proof}
    Replacing $\|d_1  + d_2\|_{\infty}$ with $\|\newton(w, \mu)\|$
    in the \Cref{thm:fdinfbound} lower bound yields
    \begin{align} \label{eq:div:bound}
 h(w, \hat w(\mu))\ge \frac{\| \newton(w, \mu) \|^2}{1+\| \newton(w, \mu)\|},
\end{align}
which proves the claim.
  \end{proof}
\end{cor}
%\begin{cor}~\label{cor:stepboundfromDiv}
%For $\cp > 0$ and $w \in \inter \coneName$, let $d = \newton(w, \cp)$.
%  If $h(w, \hat w(\cp)) \le \beta$, then
%    $\|d\| \le \frac{1}{2}(\beta+\sqrt{\beta+4}\sqrt{\beta})$.
%  Further, if $\beta \le \frac{1}{2}$ then $\|d\| \le 1$. 
%  \begin{proof}
%    Substituting $\|d_1  + d_2\|_{\infty}$ with $\|d\|$
%    into the \Cref{thm:fdinfbound} lower bound yields
%    \begin{align} \label{eq:div:bound}
% \beta \ge h(w, \hat w(\mu))\ge \frac{\| d\|^2}{1+\| d\|}.
%\end{align}
%    Solving for $\|d\|$  and noting
%    the right-hand-side is an increasing function of $\|d\|$ proves the claim.
%  \end{proof}
%\end{cor}
\noindent The inequalities of this section bear
strong resemblance to inequalities~\cite[Theorems 4.1.7--8]{nesterov2003introductory}
derived for \emph{self-concordant barrier functions}, standard
objects in IPM analysis.
We will elaborate on this connection in \Cref{sec:energyInter}.

%\begin{rem}
%  For a twice differentiable and strictly convex function $f : \algName \rightarrow \mathbb{R}$, 
%let $\| d \|_{x} := \sqrt{ \langle d, \nabla^2 f(x) d \rangle   }$ and 
%  $\hat h(x, y) := \langle \nabla f(y) - \nabla f(x) , x-y \rangle$.
%Standard analysis of interior-point methods employs an $f$ that is also \emph{self-concordant}, 
%a property that implies the following inequalities:
%  \begin{align*}
%\frac{\|x-y\|^2_{x}}{1 + \|x-y\|_{x}} \le \hat h (x,y) \le \frac{\|x-y\|^2_{x}}{1 - \|x-y\|_{x}}.
%  \end{align*}
% See, e.g.,~\cite[Theorems 4.1.7--8]{nesterov2013introductory}. These inequalities
%bear a striking similarity to Theorem~\ref{thm:fdinfbound} and
%are used in interior-point analysis to, for instance, bound the distance
% to the central path~\cite[Theorem 2.2.3]{renegar2001mathematical}.
%\end{rem}
\subsection[]{Quadratic convergence of Newton's method}
We have seen that  the Newton direction bounds the reduction in
divergence (\Cref{thm:newton}).  Divergence in turn bounds the size
of a full Newton step (\Cref{cor:stepboundfromDiv}).  Combining these
results proves quadratic convergence of the sequence $w_0, w_1, \ldots, w_m$
generated by Newton's method.
\begin{thm}~\label{thm:newtonConv}
For $\cp > 0$ and $w_0 \in \inter \coneName$, recursively define $w_i$
  via the iterations $w_{i+1} =
  Q(w_i^{1/2}) \exp(\newton(w_i, \mu))$.  If $h(w_0, \hat w(\cp)) \le \beta \le \frac{1}{2}$, then
    $h(w_i, \hat w(\cp)) \le   \beta^{2^i}$.
  \begin{proof}
    Let $h_i = h(w_i, \hat w(\cp))$ and $d_i =\newton(w_i, \mu)$.  
      Make the inductive hypothesis that $h_i \le 1/2$.
     Then $\|d_i\| \le 1$ by \Cref{cor:stepboundfromDiv}, implying
    $h_{i+1} \le \frac{1}{2} h_i \|d_i\|_{\infty}^2$ by \Cref{thm:newton}~\ref{item:thm:newton:full},
    which shows $h_{i+1} \le 1/2$. Since $h_0 \le 1/2$ by assumption,
    we conclude that both $h_i \le 1/2$  and $\|d_i\| \le 1$
    hold for all $i$. Further, for all $i$,
  \[
    h_{i+1} \le \frac{1}{2} h_i \|d_i\|_{\infty}^2  \le   \frac{1}{2} h_i \|d_i\|^2 \le  \frac{1}{2}  ( \|d_i\|+1) h^2_i,
  \]
    where the last inequality is~\eqref{eq:div:bound}.
    Since $\|d_i\| \le 1$, we have $h_{i+1} \le h^2_i$.
    Hence, $h_i \le (h_0)^{2^i} \le \beta^{2^i}$.
  \end{proof}
\end{thm}

Combining this with our previous bounds relating divergence and geodesic distance
(\Cref{lem:divdistBound}) leads to a proof of the centering lemma, 
reproduced below.
\newtonUpdate*
\begin{proof}
  By \Cref{lem:divdistBound}, we conclude that $h(w_0, \hat w(\mu)) \le \beta \le \frac{1}{2}$.
  By \Cref{thm:newtonConv}, this implies that
    $h(w_{i}, \hat w(\mu)) \le  \beta^{2^i}$,
  which, since ${\delta(w_{i}, \hat w(\mu))}^2 \le h(w_{i}, \hat w(\mu))$, proves the claim.
\end{proof}

\subsection{Energy interpretation and self-scaled barriers}~\label{sec:energyInter}
We conclude this section by highlighting connections with the literature.  This is strictly not needed for our analysis,
but helps put our work into context. 
The main object of study is the \emph{energy functional} $E(\gamma)$,
defined on smooth curves $\gamma : [0, 1] \rightarrow \inter\coneName$
via
\[
E(\gamma(t)) := \int^1_0 \|\gamma'(t)\|^2_{\gamma(t)} dt,
\]
where $\|v\|^2_{u}  = \langle v, Q(u)^{-1} v \rangle$.
In other words, energy is defined
by replacing $\|\gamma'(t)\|_{\gamma(t)}$
with $\|\gamma'(t)\|^2_{\gamma(t)}$
in the arc-length integral (\Cref{sec:geo}).

We next show that divergence $h(u, v)$ is precisely the energy of the
line-segment connecting $u$ and $v$. 
This immediately implies the \Cref{lem:divdistBound}  inequality
$\delta(u, v)^2 \le h(u, v)$ 
given that $\delta(u, v)^2 \le E(\gamma)$ holds
for any curve $\gamma(t)$ connecting $u$ and $v$; see~\cite[Chapter 9, Lemma 2.3]{do1992riemannian}.

\begin{prop}
For $u, v \in \inter\coneName$,  let $\ell(t):= u + t (v-u)$. Then
 $h(u, v) = E(\ell(t))$, i.e.,
\begin{align}~\label{eq:energyIntegral}
h(u, v) = \int^1_0 \langle v- u, Q(\ell(t))^{-1} (v-u) \rangle dt.
\end{align}
\begin{proof}
By definition, $h(u, v) = \langle u, v^{-1}\rangle + \langle u^{-1}, v \rangle - 2n$.
Rearranging shows $h(u, v) = \langle v - u, u^{-1} - v^{-1}  \rangle$.
Since $-Q(z)^{-1}$ is the Jacobian of the inverse map $z \mapsto z^{-1}$~\cite[Proposition II.3.3]{faraut1994analysis}, we can also write
\[
v^{-1} -  u^{-1}  = \int^1_0 -Q(\ell(t))^{-1} (v-u) dt.
\]
Hence, 
$h(u, v) = \langle v - u, u^{-1} - v^{-1}  \rangle = \int^1_0 \langle v- u, Q(\ell(t))^{-1}(v-u) \rangle dt$,
as claimed.
\end{proof}
\end{prop}

%\begin{lem}
%\[
%(\sup \{ \alpha : x - \alpha p \in \coneName\})^{-1} = \inf \{ \beta >= 0 : \beta x - p \}
%\]
%\begin{proof}
%Suppose both are attained. Then since $x \in \inter \coneName$,
%we must have that $\alpha > 0$. This shows that $\frac{1}{\alpha} x - p \in \coneName$,
%implying that
%\[
%\frac{1}{\alpha_*} \ge \beta_* 
%\]
%On the other hand
%\[
% x - \frac{1}{\beta_*} p \in \coneName 
%\]
%implying that
%\[
%\frac{1}{\beta_*} \le \alpha_*
%\]
%So
%\[
%\beta_* \ge \frac{1}{\alpha_*}
%\]
%we conclude $\beta_* = \frac{1}{\alpha_*}$.
%Suppose $\beta_* = 0$. Then $-p \in \coneName$, which shows that
%$\alpha_* = \infty$. On the other hand, if $\alpha_* = \infty$, 
%we can construct a sequence of $\beta \ge 0$ converging to zero,
%showing that $\beta_* = 0$.
%\end{proof}
%\end{lem}

In view of this result, we can bound $h(u, v)$ by bounding
$Q(\ell(t))^{-1}$.  For this, we use  standard Hessian bounds 
for \emph{self-scaled
barrier functions}~\cite{nesterov1998primal}, which
generalize $f(u):=\log \det u^{-1}$ and are central in IPM analysis over \emph{self-scaled cones}.
Specifically, we interpret $Q(u)^{-1}$ as the Hessian of
$f(u)$ and invoke~\cite[Theorem
4.1]{nesterovTodd1997self}; see~\cite{nesterovTodd1997self, hauser2002self} for the 
definition of self-scaled barriers and proof that $f(u)$ is self-scaled. 
\begin{prop}\label{prop:selfScaledBounds}
For $u, v \in \inter \coneName$, let $\Delta = Q(u)^{-1/2}(u-v)$.
If $\|\Delta\|_{\infty} < 1$, then
\begin{align}\label{eq:boundsFromSelfScaled}
 \frac{\|\Delta\|^2}{1 +\|\Delta\|_{\infty} }  \le     h(u, v) \le  \frac{\|\Delta\|^2}{1 -\|\Delta\|_{\infty} }.
\end{align}
\begin{proof}
Let $H(z) := Q(z)^{-1}$ 
and $\sigma_{z}(p) := \inf\{ \beta \ge 0 : \beta z - p \in \coneName\}$.  
Then~\cite[Theorem 4.1]{nesterovTodd1997self} states 
\begin{align*}
 \frac{1}{(1 + t \sigma_{u}(-p) )^2}  H(u) \preceq   H(u-tp) \preceq \frac{1}{(1 -t \sigma_{u}(p) )^2}  H(u)
\end{align*}
for all $t \in [0, 1/\sigma_{u}(p))$, where we take $1/\sigma_{u}(p) = +\infty$ if $\sigma_{u}(p) = 0$.   Taking $p = u-v$ and observing % chktex 9
\[
\sigma_{u}(p) \le  |\lambda_{\max}(Q(u)^{-1/2} p)| \le \|\Delta\|_{\infty}, \qquad \sigma_{u}(-p) \le |\lambda_{\min}( Q(u)^{-1/2} p )| \le \|\Delta\|_{\infty},
\]
gives, for all $t \in [0, 1/\|\Delta\|_{\infty})$, the bounds % chktex 9
%we conclude $\sigma_{u}(-p) \le \|\Delta\|_{\infty}$ and $\sigma_{u}(p)\le \|\Delta\|_{\infty}$.  Hence, 
\begin{align*}
 \frac{1}{(1 + t \|\Delta\|_{\infty} )^2}  H(u) \preceq   H(u+t(v-u)) \preceq \frac{1}{(1 -t \|\Delta\|_{\infty} )^2}  H(u).
\end{align*}
When  $\|\Delta\|_{\infty} < 1$, we can substitute each bound into the energy integral~\eqref{eq:energyIntegral} and apply the identities
\[
\int^1_0 \frac{dt}{(1 +t \|\Delta\|_{\infty}  )^2} = \frac{1}{1+\|\Delta\|_{\infty}}, \qquad
\int^1_0 \frac{dt}{(1 -t \|\Delta\|_{\infty} )^2} = \frac{1}{1-\|\Delta\|_{\infty}},
\]
to conclude that
\[
 \frac{ \langle H(u) p, p \rangle }{1 +\|\Delta\|_{\infty}}  \le   h(u, v) \le \frac{ \langle H(u) p, p \rangle  }{1 -\|\Delta\|_{\infty}}.
\]
Since $\|\Delta\|^2 = \langle H(u) p, p \rangle$, the claim follows.
\end{proof}
\end{prop}
When applied to $h(w, \hat w(\mu))$, the bounds~\eqref{eq:boundsFromSelfScaled} are similar to those from \Cref{thm:fdinfbound},
but not equivalent.  In particular,~\eqref{eq:boundsFromSelfScaled} requires the unknown quantity 
$\hat w(\mu)$ to construct $\Delta$,
whereas \Cref{thm:fdinfbound} uses  the Newton direction $\newton(w, \mu)$.
%the direct-summands $d_1$ and $d_2$
Further,~\eqref{eq:boundsFromSelfScaled} 
does not preserve the symmetry $h(w, \hat w) = h( w^{-1},\hat w^{-1})$, 
as replacing $\Delta = Q(w)^{-1/2}(w-\hat w)$ with $\Delta = Q(w^{-1})^{-1/2}(w^{-1} - \hat w^{-1})$
leads to  different bounds.
%JIn fact, we can interpret \Cref{thm:fdinfbound} as an "average"
%Jof these bounds in the sense that the vectors $d_1$ and $d_2$
%Jfrom this theorem satisfy
%J\[
%Jd_1 =  \proj_{\mathcal{L}^\perp} Q(w^{1/2})( \hat w^{-1} -   w^{-1}), \qquad
%Jd_2 = \proj_{\mathcal{L}_w} Q(w)^{-1/2}(\hat w - w).
%J\]

We also note that~\eqref{eq:boundsFromSelfScaled} still holds if  $\|\Delta\|_{\infty}$ is replaced
with $\|\Delta\|$. With this replacement, 
it is a special case of~\cite[Theorem 4.1.7--8]{nesterov2003introductory},
which holds for arbitrary \emph{self-concordant} functions, a 
superset of self-scaled functions that are central in IPM analysis
over general convex sets.

\section{Long-step algorithm}\label{sec:prac}
\begin{figure}~\label{fig:longstep}
  \qquad
  \begin{minipage}[t]{.5\linewidth}
  \begin{algorithm}[H]
  \SetAlgoLined\DontPrintSemicolon{}
  \SetKwFunction{algo}{longstep}\SetKwFunction{proc}{center}
  \SetKwProg{myalg}{Algorithm}{}{}
  \myalg{\algo{$w_0, \mu_0, \mu_f$, $\epsilon$}}{%chktex  
  \vspace{.1cm}
  $\mu \leftarrow \mu_0$, $w \leftarrow w_0$ \\
  \While{$\mu  > \mu_f$} 
  {%chktex
    $w \leftarrow \proc(w, \cp, \alpha)$ \\
    $\cp \leftarrow \inf \{ \cp > 0 :  h_{ub}(w, \cp) \le \beta \}$
  }
    \Return{$\proc(w, \cp, \epsilon)$}  \\
  }{} 
  \end{algorithm} 
  \end{minipage}
  \begin{minipage}[t]{.5\linewidth}
    \begin{procedure}[H]
  \SetKwProg{myproc}{Procedure}{}{}
  \myproc{\proc{$w_0, \mu, \epsilon$}}{%chktex
  \vspace{.1cm}
  $w \leftarrow w_0$ \\
  \While{$h_{ub}(w, \mu)  > \epsilon$}{%chktex
    $d \leftarrow  \newton(w,  \mu)$ \\
    $\gamma \leftarrow \max\{1, \frac{1}{2\theta} \|d\|^2_{\infty}\}$ \\
    $w \leftarrow  Q(w^{1/2}) \exp(\frac{1}{\gamma} d)$ \\
  }
  \Return{} $w$}
\end{procedure}
  \end{minipage}
  \caption{A long-step algorithm (left) and  centering
  procedure (right). The parameters $\beta$ and $\alpha$ control distance to the
  central path and  $\theta$ the transition to damped Newton steps. 
  The algorithms globally convergence on all inputs if $1 > \theta > 0$ 
  and $\beta > \alpha > 0$. }~\label{fig:long}
\end{figure}

When proving the convergence of {\tt shortstep} (\Cref{alg:barrier}), we
established results that suggest an alternative algorithm.  This alternative uses our divergence upper-bound
(\Cref{thm:fdinfbound}) to loosely track the central path and 
damped Newton steps to ensure that divergence strictly decreases (\Cref{thm:newton}).
We state this algorithm in~\Cref{fig:long} using
the following notation for the divergence upper-bound:
\[
  h_{ub}(w, \mu) = \begin{cases} \frac{\|\newton(w, \mu)\|^2}{1-\|d_1(w, \mu) + d_2(w, \mu)\|_{\infty}} & \|d_1 + d_2\|_{\infty} < 1  \\
                                  \infty & otherwise,
  \end{cases} 
\]
 (Here  $d_1(w, \mu)$ and $d_2(w, \mu)$  denote the direct-summands of the Newton direction
 $\newton(w, \mu)$; see \Cref{prop:newProj}.)
We name this algorithm {\tt longstep} 
in reference to classical long-step IPMs~\cite{wright1997primal},
which also loosely track the central path.

The next theorem shows that {\tt longstep} is  \emph{globally
convergent}, i.e.,  it can be initialized arbitrarily. 
To prove this, we exploit the fact that the sublevel
sets of divergence $h$ are compact, which implies positive lower bounds on certain
progress measures.  This theorem also shows scale invariance (Section~\ref{sub:affineinvar}).
This follows from~\Cref{prop:affine} given that the step-size $\gamma^{-1}$ and 
divergence bound $h_{ub}$ depend only on the eigenvalues
of $d_1 + d_2$ and $d_1 - d_2$.
\begin{thm}
  If $1 > \theta > 0$ and $\beta > \alpha > 0$, then
  the algorithm {\tt longstep} and its subroutine {\tt center}~(Figure~\ref{fig:long}) have the following properties.
  \begin{enumerate}[label= (\alph*)]
    \item\label{thm:long:conv} For all inputs $w_0 \in \inter \coneName$ and
      $(\mu_0, \mu_f, \epsilon) > 0$, {\tt longstep}  terminates and returns
      $w$ satisfying $h(w, \hat w(\mu)) \le \epsilon$ for $\mu \le \mu_f$.  Further, it
      monotonically decreases $\mu$.
    \item\label{thm:long:cent} For all inputs $w_0 \in \inter \coneName$ and
      $(\mu, \epsilon) > 0$, {\tt center}  terminates and returns $w$ satisfying
      $h(w, \hat w(\mu)) \le \epsilon$.  Further, it monotonically decreases
      $h(w, \hat w(\mu))$.
    \item\label{thm:long:aff} Both {\tt center} and {\tt longstep} are scale invariant.
  \end{enumerate}
  \begin{proof}
    To prove statements~\ref{thm:long:conv}-\ref{thm:long:cent}, we first show compactness
    of the set
\[
   S(\zeta) := \{ (w, \mu) : h(w, \hat w(\mu)) \le \zeta, \mu_f \le \mu \le \mu_0  \}.
\]
     It is closed because $(w, \mu) \mapsto h(w, \hat w(\mu))$ is
    continuous. To see it is bounded, note that the eigenvalues of $\hat w(\mu)$ and
    $\hat w(\mu)^{-1}$ are bounded below by some $c > 0$ on $\mu_f \le \mu \le
    \mu_0$, implying that 
\[
    \zeta  \ge  h(w, \hat w(\mu)) \ge c \langle e,  w +   w^{-1}\rangle  - 2n \ge c \|w\|_1 -2n
    \]
    when $(w, \mu) \in S(\zeta)$.  Hence, if $(w, \mu) \in S(\zeta)$
    then $\|w\|_1 + |\mu|$ is bounded,
    implying $S(\zeta)$ is compact.

    To prove~\ref{thm:long:cent},  let $\zeta = h(w_0, \hat w(\mu))$.
    Let $\Delta(w, \mu)$ denote the decrease in $h$ after one Newton step from $w$,
    i.e., 
    \[
    \Delta(w, \mu)  =  h(\hat w(\mu), w) -  h(\hat w(\mu), w')
    \]
     where   $w' =  Q(w^{1/2}) \exp(\frac{1}{\gamma} d)$.
    After $N$ steps, $h(w, \hat w(\mu)) \le h(w_0, \hat w(\mu)) - \Delta_* N$,
    where
    \[ 
    \Delta_* := \inf_{w, \mu} \{ \Delta(w, \mu) :  h_{ub}(\hat w(\mu), w) \ge
    \epsilon, (w, \mu) \in S(\zeta)  \}.
    \]
    Compactness of $S(\zeta)$ implies $\Delta_*$ is attained, which implies $\Delta_*  > 0$  by our step-size rule and \Cref{thm:newton}.
    Since  $h \ge 0$, we conclude that {\tt center} must terminate before
    $\Delta_* N > h(w_0, \hat w(\mu))$.

    To prove statement~\ref{thm:long:conv}, note that $\mu \le k_*^{-M} \mu_0$
    after $M$ iterations, where
    \[
    k_* := \inf_{w, \mu , k}\{ k \ge 1: (w, \mu) \in  S(\alpha),  h_{ub}(w,  \mu) \le \alpha,  h_{ub}(w, (1/k) \mu) = \beta \}.
    \]
    Compactness of $S(\alpha)$ implies $k_*$ is attained, which implies that
    $k_* > 1$ since $\beta > \alpha$.  This implies $\mu < \mu_f$ eventually holds,
    implying termination of {\tt longstep}.

    Finally, statement~\ref{thm:long:aff}  follows from~\Cref{prop:affine} and
    \Cref{lem:mapInvar}, given that $\gamma$ and $h_{ub}$, viewed as  functions of
    $d_1$ and $d_2$, are invariant under transformation by an orthogonal
    automorphism $M$, i.e., $\gamma(d_1, d_2) = \gamma(M d_1, Md_2)$
    and $h_{ub}(d_1, d_2) = h_{ub}(M d_1, Md_2)$.
  \end{proof}
\end{thm}

We close this section with practical matters related to implementation.
Specifically, we show how to efficiently evaluate the divergence bound $h_{ub}(w, \mu)$ for
fixed $w$,   how to find the Newton direction using a least-squares technique,
how to evaluate geodesic updates without computation of $w^{1/2}$,
 and how to construct feasible points for the primal-dual pair~\eqref{sdp:main}.

\subsection{Evaluating divergence for $\mu$-selection}~\label{sec:muselect}
For fixed $w$, the divergence bound $h_{w, ub}(\mu) := h_{ub}(w, \mu)$
has a simple formula that admits efficient selection of $\mu$
at each iteration of {\tt longstep}. To evaluate the formula,
we only need to know $\mu$ and quantities involving the vector
\[
  g_w:= \proj_{\mathcal{L}^{\perp}_w}{Q(\sqrtminvw)}x_0  + \proj_{\mathcal{L}_w}{Q(\sqrtmw)}s_0,
\]
where we recall that $\mathcal{L}_w := \{ Q(\sqrtminvw) u :  u \in \mathcal{L}
\}$ and  $\mathcal{L}_w^{\perp} = \{ Q(\sqrtmw) u : u \in \mathcal{L}^{\perp}
\}$.

\begin{prop}
For $w \in \inter \coneName$, let $g_w$ have minimum and maximum eigenvalues $\lambda_{\min}$ and $\lambda_{\max}$. 
  Let $k(\mu)  = \min( \frac{1}{\sqrt \mu}\lambda_{\min},   2-\frac{1}{\sqrt \mu}\lambda_{\max})$.
  Then, for all $\mu > 0$,
  \[
  h_{w, ub}(\mu) = \begin{cases}  \frac{\frac{1}{\mu} \|g_w\|^2 - 2 \frac{1}{\sqrt{\mu}} \tr g_w + n }{k(\mu)}  & k(\mu) > 0 \\
                                  \infty & \mbox{otherwise}.
  \end{cases}
  \]
  \begin{proof}
    Let $d=\newton$ and let $d_1$ and $d_2$ be as in Proposition~\ref{prop:newProj}. Suppose 
    that $h_{ub}(w, \mu)$ is finite, i.e., $1-\|d_1 + d_2\|_{\infty} > 0$.
    Then we have that
    \[
      h_{ub} = \frac{\|d\|^2}{1-\|d_1 + d_2\|_{\infty}}, \qquad d_1 + d_2 =  \frac{1}{\sqrt{\mu}} g_w - e.
  \]
     Hence, $\|d_1 + d_2\|_{\infty}$ is the max of  $1-\frac{1}{\sqrt \mu}\lambda_{\min}$
     and $\frac{1}{\sqrt \mu}\lambda_{\max} - 1$.
     The claimed denominator $k(\mu)$ follows using the identity
      \[
        1 - \max(1-a, b-1) = 1 + \min(a-1, 1-b) =  \min(a, 2-b).
      \]
      The identity for $\|d\|^2$ follows by  expanding
      $\|\frac{1}{\sqrt \mu}g_w - e\|^2$ and observing that $\|d\| = \|d_1 + d_2\|$.
  \end{proof}
\end{prop}

%Related feasibility results for the Nesterov-Todd direction are given 
%in~\cite[Lemma~4.2]{roos2006full}. Note this paper also has similar looking quad-convergence bounds.

\subsection{Newton direction via least squares}\label{sec:leastsquares}
Interior-point methods typically  find  search directions by solving
least-squares problem of the form
\[
  \mbox{minimize}_y \;\; \frac{1}{2} y^T  A^* W(x, s) A y  - f^{T}y \mbox{ subject to } By = g,
\]
where $W(x, s)$ is a positive-definite weighting matrix induced by the current iterate $(x, s)$
and  $(A, B, f, g)$ are parameters induced by the affine constraints
$x_0 + \mathcal{L}$ and $s_0 + \mathcal{L}^{\perp}$.
Equivalently, they solve linear systems of the form 
\begin{align*}
 \begin{bmatrix} 
  A^*W(x, s) A & B^* \\
  B & 0 \\
  \end{bmatrix}
  \begin{bmatrix}
    y \\ z
  \end{bmatrix}
  =
  \begin{bmatrix}
    f \\ g
  \end{bmatrix}
\end{align*}
for which specialized algorithms exist (e.g.,~\cite{lawson1995solving}).
Such a system can also yield the Newton direction $\newton(w, \mu)$.
This, of course, is not surprising given its construction 
via orthogonal projection~(\Cref{prop:newProj}).
Nevertheless, we give this system explicitly  for affine constraints  of the 
form:
\begin{align}~\label{eq:affine}
  s_0 + \mathcal{L}^{\perp} =  \{ c - Ay :   By = g, \;\; y \in \mathbb{R}^m \}, \qquad
  x_0 + \mathcal{L} =  \{ x \in \algName : \exists z \in \mathbb{R}^d \;\; A^*x + B^*z = b \},
\end{align}
where $(y, z) \in \mathbb{R}^{m} \times \mathbb{R}^{d}$ denote
additional variables, $A : \mathbb{R}^m \rightarrow \algName$ and $B : \mathbb{R}^m \rightarrow \mathbb{R}^d$
are linear maps with adjoint operators $A^* : \algName \rightarrow \mathbb{R}^m$
and $B^* : \mathbb{R}^{d} \rightarrow \mathbb{R}^m$, and $(b, g, c) \in \mathbb{R}^{m}  \times\mathbb{R}^{d}\times \algName$
are fixed parameters.

In this notation, the Newton direction becomes the $d$ that
for some $(y, z)$ solves 
  \begin{align} \label{eq:elim}
  A^* ( Q(w^{1/2})(e+d)   )=   \frac{1}{\sqrt{\mu}}b - B^*z, \;\;\;
    Q(w^{-1/2})(e-d)   =      \frac{1}{\sqrt{\mu}}c - Ay,   \;\;\; By   = \frac{1}{\sqrt \mu}g.
\end{align}
Eliminating $d$ and using $Q(w) = Q(\sqrtmw) Q(\sqrtminvw)^{-1}$
yields a system with the desired form.  Note by modifying the right-hand-side of this
system, we can also construct the direct-summands $d_1$ and $d_2$ of \Cref{prop:newProj}.
\begin{prop}~\label{thm:schur}
  For $w \in \inter \coneName$ and $\mu >0$,
  let  $(y, z) \in \mathbb{R}^m \times \mathbb{R}^d$ solve the least-squares system
    \[
\begin{bmatrix}
  \label{eq:schur}
  A^* Q(w) A & B^* \\
  B & 0 \\
\end{bmatrix} \begin{bmatrix}
                y \\
                z
               \end{bmatrix} =
\begin{bmatrix}
  \frac{1}{\sqrt{\mu}}(b + A^* Q(w) c)  -2A^* w       \\
  \frac{1}{\sqrt{\mu}}g
 \end{bmatrix}.
 \]
  Then,  the Newton direction  satisfies
  $\newton(w, \mu) = e - Q(\sqrtmw)  (\frac{1}{\sqrt{\mu}}c - Ay)$.
\begin{proof}
  From the second equation of~\eqref{eq:elim}, we conclude that
  $d  =  e - Q(\sqrtmw)  ( \frac{1}{\sqrt{\mu}} c - Ay)$.
  Substituting into the first equation yields
  \begin{align*}
    \frac{1}{\sqrt{\mu}} b  - B^*z &=  A^* Q(\sqrtmw) \bigg( 2e - Q(\sqrtmw)  (  \frac{1}{\sqrt{\mu}} c - Ay) \bigg) \\
                                        &=  2A^* w -   \frac{1}{\sqrt{\mu}} A^* Q(w) c + A^* Q(w) Ay.
  \end{align*}
  Rearranging terms proves the claim.
\end{proof}
\end{prop}

%\begin{rem}
%In passing, we proved that $\newton(w, \mu) = e - Q(\sqrtmw)  (\frac{1}{\sqrt{\mu}}c - Ay)$.
%Letting $x = \sqrt\mu w$ and $s = c - \sqrt\mu Ay$ allows
%  us to interpret $\|\newton \|$ in terms of  $\|\mu e - Q(x^{1/2}) s\|$. Precisely,
%  \[
%    \mu \newton = \mu e - \sqrt\mu Q(\sqrtmw)  (c - A \sqrt\mu y) = \mu e -  Q(x^{1/2})  s.
%  \]
%  Given search directions $(\Delta x, \Delta s)$, 
%  some primal-dual IPMs use the eigenvalues of $\mu e - Q((x  + t \Delta x) ^{1/2}) s(x + t\Delta s)$ 
%  to compute a step-size $t$; see, e.g.,~\cite[Section 4]{sturm2002implementation}.
%  This can be compared with our use of the spectral norm  $\|d_{N}\|_{\infty}$.
%\end{rem}

\subsection{Evaluation of geodesic updates}
For $w \in \inter \coneName$ and $v \in \algName$, let $g(w, v) := Q(w^{1/2}) \exp( Q(w^{1/2}) v)$.
By Proposition~\ref{thm:schur}, we see that 
$\newton(w, \mu) = e + Q(w^{1/2})v$ for a particular point $v\in \algName$.
Hence, the geodesic update $Q(w^{1/2}) \exp (\newton)$
satisfies, for particular  $\kappa > 0$,  the equation
\begin{align}~\label{eq:geoform}
  Q(w^{1/2}) \exp (\newton) =  \frac{1}{\kappa} g(w, v).
\end{align}
We next show that $g(w, v)$ can be computed without constructing the square
root $w^{1/2}$. Letting $z = Q(w^{1/2}) v$, the key idea is expressing the power series
of $\exp(z)$ in terms of $Q(z)$ and applying the identity $Q(z) = Q(w^{1/2}) Q(v) Q(w^{1/2})$
from \Cref{sec:appendix}.

\begin{prop}~\label{prop:geopower}
  If $w \in \coneName$ and $v \in \algName$, then $g(w, v) =
  \sum^{\infty}_{n=0} \frac{1}{(2n)!}( Q(w) Q(v))^{n}  (w +   \frac{1}{2n+1}
  Q(w)v)$.
  \begin{proof}

For arbitrary $z$, we have that $z^{2n} = Q(z)^{n} e$ and $z^{2n+1} = Q(z)^{n} z$, which implies that
    \begin{align}\label{eq:geoform1}
      Q(w^{1/2})   \exp(z) = Q(w^{1/2}) \sum^{\infty}_{i=0} \frac{1}{(2n)!} Q(z)^{n} (e +   \frac{1}{2n+1}  z).
    \end{align}
    For $z  = Q(w^{1/2}) v$, we have, using $Q(z)^{n} = (Q(w^{1/2})Q(v)Q(w^{1/2}))^{n}$, that
    \begin{align}\label{eq:geoform2}
      Q(w^{1/2}) Q(z)^{n} = (Q(w) Q(v))^{n}  Q(w^{1/2}).
    \end{align}
    Substituting~\eqref{eq:geoform2} into~\eqref{eq:geoform1} proves the claim.
%
%    Q(w^{1/2})Q(d)Q(w) Q(d)Q(w) Q(d)Q(w^{1/2})
%
%    (Q(w) Q(d))^n  Q(w^{1/2}) (e +   Q(w^{1/2}) d) 
%    (Q(w) Q(d))^n  (w +   Q(w) d)
%
  \end{proof}
\end{prop}

An alternative formula for $g(w, v)$  is available when $\algName$ is \emph{special}, i.e.,
if the product operation $x \circ y$ satisfies  $x\circ y = \frac{1}{2}(xy + yx)$
for an associative  product $xy$. 
An example of a special algebra is the set of symmetric
matrices with product $\frac{1}{2}(XY+YX)$, where $XY$ denotes ordinary matrix multiplication.
For any special algebra,
the quadratic representation satisfies $Q(x)y = x y x$ for all $x, y \in \algName$.
This fact allows us to compute  $g(w, v)$ by evaluating
  $\exp (wv) := \sum^{\infty}_{d=0} \frac{1}{d!} (wv)^d$, i.e., the
  exponential map induced by the associative product, at the point $wv$.
\begin{prop}~\label{prop:geospecial}
  If $\algName$ is special then $g(w, v)=  \exp (w v) w$
  for  all $v \in \algName$ and $w \in \inter \coneName$.
  \begin{proof}
    Let $z = Q(w^{1/2}) v$.  Since $z = w^{1/2} v  w^{1/2}$, we have that
    \[
      w^{1/2} z^d w^{1/2} = w^{1/2} w^{1/2} v (w v)^{d-1}  w^{1/2} w^{1/2} =  (wv)^d w
    \]
    Hence, 
    \[
    Q(w^{1/2})  \exp(Q(w^{1/2}) v) = w^{1/2}  \left( \sum^{\infty}_{d=0}  \frac{1}{d!} z^d \right) w^{1/2}=   \left(\sum^{\infty}_{d=0}\frac{1}{d!} (w v)^d \right) w = \exp (w v) w.
    \]
  \end{proof}
\end{prop}

Note that for symmetric matrices,  $\exp(WV)$ is the usual
matrix exponential evaluated at the matrix product $WV$. One can evaluate the
matrix exponential using a power series or Pade
approximation~\cite{moler1978nineteen}. In total, the entire
evaluation of $Q(w^{1/2}) \exp(\newton)$ can be done without eigenvalue decomposition
or square roots.

\begin{rem}
  The quantity $g(w, v) := Q(w^{1/2}) \exp(Q(w^{1/2}) v)$ can be
  written using the manifold exponential map (\Cref{sec:manifold}) as
  \[
    g(w, v) = \Exp_w( Q(w) v). 
  \]
    For the algebra of
  symmetric matrices $(\algName = \mathbb{S}^n)$,
  it's known (e.g.,~\cite{sra2015conic}) that $\Exp_W$ satisfies
  \[
    \Exp_W(Z)  =  \exp(Z W^{-1} ) W.
  \]
  This provides  an alternative proof of \Cref{prop:geospecial}
  for the special case of $\algName = \mathbb{S}^n$.  Precisely, taking $Z = Q(W) V = WVW$, we deduce that
  \[
    g(W,V) = \Exp_W( Q(W) V)  =  \exp( (W V W) W^{-1} ) W = \exp( W V ) W,
  \]
  as claimed.
\end{rem}

%\subsection{Evaluation of norms}
%\begin{lem}
%  \[
%    \|z\|^2_{\infty} =   \|Q( z)  \|_{\infty}
%  \]
%  Hence, letting $z = Q(w^{1/2})s$
%\[
%  \|Q(w^{1/2})s\|^2_{\infty} =  \|Q(w^{1/2})Q(s)Q(w^{1/2})   \|_{\infty} = \lambda_{max} (Q(w)Q(s))
%\]
%\end{lem}

\subsection{Feasible points}
Since the presented algorithms update $w$ along geodesics,  the point $\sqrtcp(w, w^{-1})$ only
satisfies the affine constraints of the primal-dual pair~\eqref{sdp:main} in
the limit.  Nevertheless, under a norm condition,  we can \emph{always} produce
a feasible $(x, s)$ from the Newton direction $\newton(w,\mu)$.
\begin{prop}\label{prop:feas}
For $w \in \inter \coneName$ and $\mu > 0$, let
  $d = \newton(w, \mu)$ and 
  \[
    x =  \sqrtcp Q(\sqrtmw)(e + d), \qquad s =  \sqrtcp Q(\sqrtminvw)(e - d).
  \]
  If $\|d\|_{\infty} \le 1$, then $(x, s)$ is feasible for~\eqref{sdp:main}.
  \begin{proof}
    By definition of the Newton direction (Definition~\ref{defn:NewtonDir}),
    it holds that $x \in x_0 + \mathcal{L}$ and
    $s \in s_0 + \mathcal{L}^{\perp}$.  Further, since $\|d\|_{\infty} \le 1$,
    we have that $e\pm d \in \coneName$. Finally, $x, s \in \coneName$
    given that $Q(z)y  \in \coneName$ for all $ z \in \algName$ and $y \in \coneName$.
  \end{proof}
\end{prop}
\noindent In light of \Cref{sec:ntcomp}, this proposition gives a
sufficient condition for feasibility of a full Nesterov-Todd step when $x = \mu s^{-1}$.
It can therefore be compared with~\cite[Lemma 3.3]{gu2011full}.

\section{Computational results}\label{sec:comp}
We provide a series of computational experiments that illustrate
key features of our algorithms and the performance of an implementation.
First, we illustrate that {\tt longstep} executes far fewer iterations than {\tt
shortstep}, despite its weaker theoretical guarantees.
We then illustrate {\tt longstep}-performance on a range of symmetric cones, 
including the exceptional cone and the psd Hermitian matrices with complex
and quaternion entries;  to our knowledge, these are  the first computational results for the
quaternion and the exceptional cone.
We  next demonstrate  global convergence of the centering procedure {\tt center}.
Finally, we compare our {\tt longstep}-implementation {\tt conex} (pronounced
CON-ex) to {\tt sdtp3}, a widely used solver that is based on the Nesterov-Todd
algorithm~\cite{toh2009sdpt3}.

For each instance, the affine constraints
are of the form~\eqref{eq:affine} but with the equality constraints $By =g$ and
associated dual variable $z$ omitted.
The operator $A : \mathbb{R}^m \rightarrow \algName$ is randomly generated. 
Unless stated otherwise, the cone $\coneName$ is the set of psd matrices $\mathbb{S}^n_{+}$,
the cost vectors are the identity, i.e., $x_0 = e$ and $s_0 = e$,
and $m = 10$.

%Results show superior performance of {\tt longstep}, mirroring the performance gap
%between  short- and long-step interior-point methods~\cite{potra2000interior}.
%Moreover, this performance depends only on $n$, the rank of the cone.
%We observe distinct convergence phases for {\tt center}. 

\subsection{Algorithm comparison}
\begin{figure}
  \begin{tabular}{cc}
  \hspace{-1.3cm}
    \includegraphics[width=.57\linewidth]{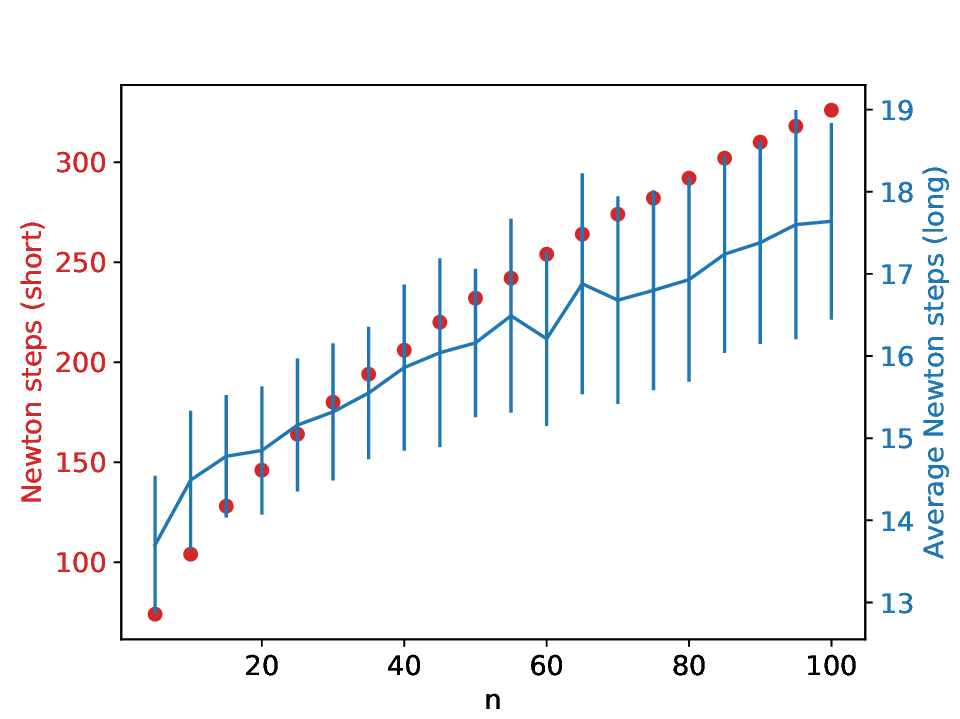}
    \includegraphics[width=.57\linewidth]{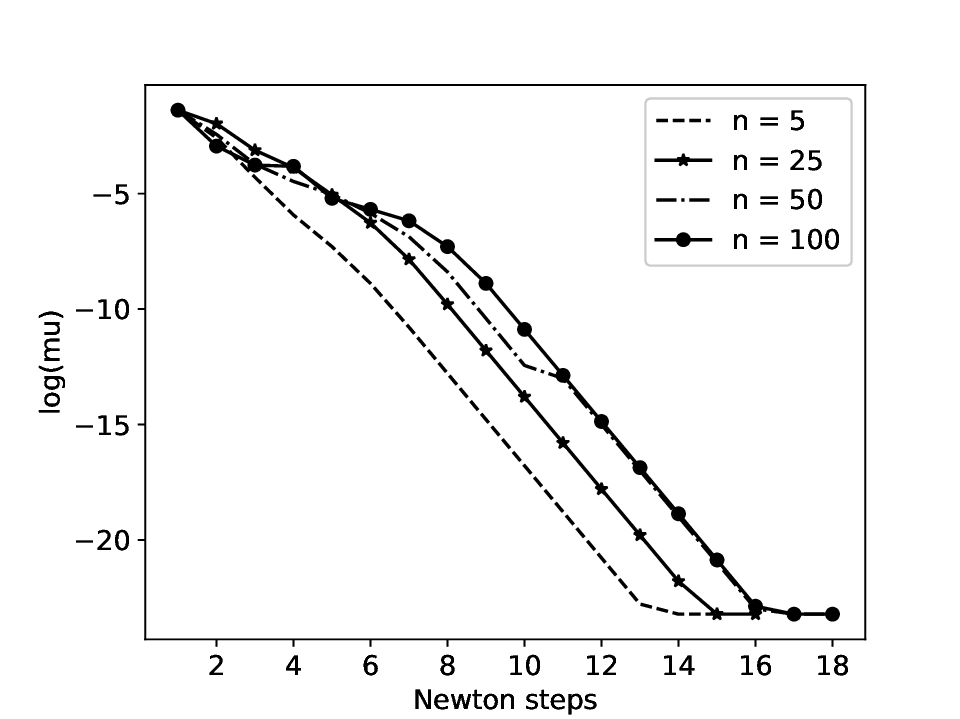}
  \end{tabular}
    \caption{Total Newton steps  vs  $n$ for {\tt shortstep} and {\tt longstep}
    (left) on random SDPs $(\coneName = \mathbb{S}_+^n)$. (Note the different scales.) Typical decrease in centering
    parameter $\mu$ for {\tt longstep} (right).}~\label{fig:shortvslong}
\end{figure}
We compare (Figure~\ref{fig:shortvslong})
the total number of Newton steps  {\tt longstep} and {\tt shortstep} execute to
update an initial centered point $\hat w(\mu )$ to $\hat w(\frac{1}{k} \mu)$
where $k= 25000^2$. For each $n$, we compute the average number of steps
executed by {\tt longstep} over twenty random problems semidefinite programs
($\coneName = \mathbb{S}^n_{+}$). The number of steps executed by {\tt
shortstep}
is independent of the problem instance, so no averaging is necessary.
As shown,  {\tt longstep} provides a significant  improvement
over {\tt shortstep}.  We also see that {\tt longstep} enters a steady-state
regime in which it reduces the centering parameter $\mu$ at a rate that is
independent of $n$.

For {\tt longstep}, we  chose a divergence bound $\beta$ that grows linearly
with $n$; specifically, we took $\beta = 100 n$ where 100 is chosen arbitrarily
and the dependence on $n$ is intended to compensate for the
$\|\newton\|^2$-dependence of the divergence upper-bound $h_{ub}$.
We chose a re-centering tolerance of $\alpha = 10$, and a final centering tolerance of
$\epsilon = \frac{1}{200}$.  For {\tt shortstep},  we selected the
centering-parameter update $k$ and the number of inner iterations $m$ using
\Cref{thm:barrier} with the parameter values $(\beta, \epsilon) =
(\frac{1}{2}, \sqrt{\frac{1}{200}})$.

\begin{figure}
  \begin{tabular}{cc}
  \hspace{-1.3cm}
    \includegraphics[width=.57\linewidth]{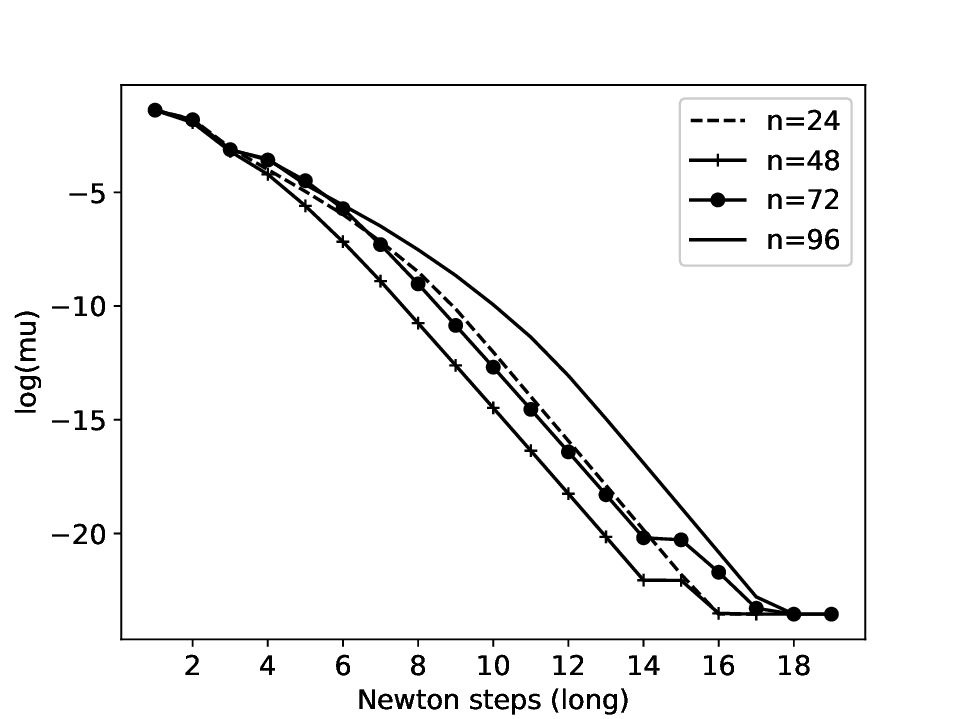}
    \includegraphics[width=.57\linewidth]{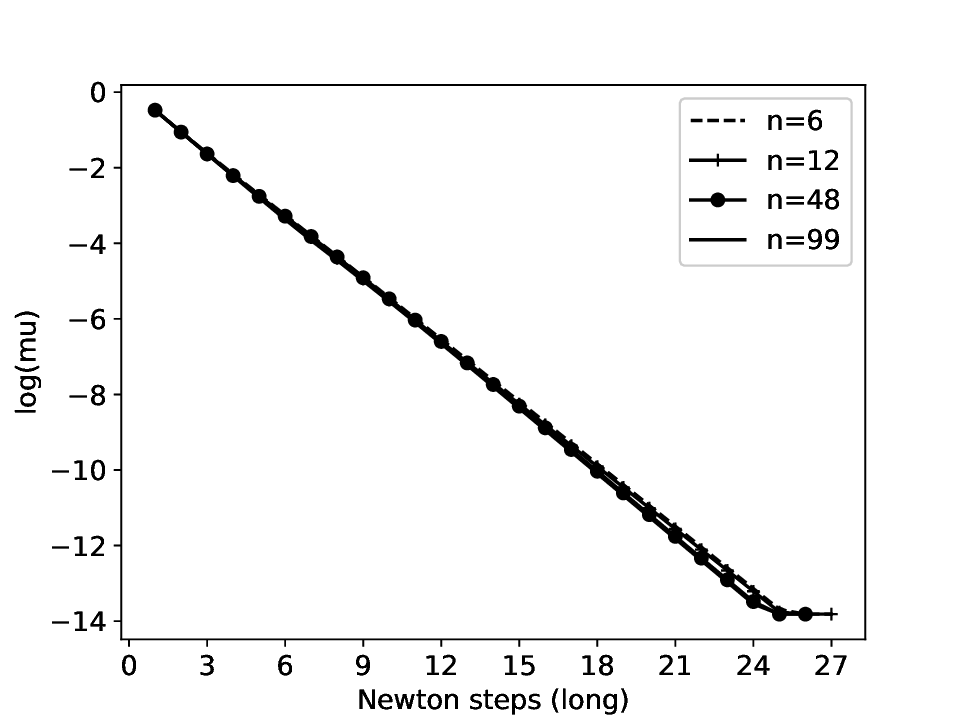}
  \end{tabular}
    \caption{Typical decrease in centering parameter $\mu$ for {\tt longstep} on a
      $p$-fold product of a special cone $\coneName_{s}$ (left) and 
      $p$-fold product of an exceptional cone $\coneName_{ex}$ (right).
      The $p$-fold product of $\coneName_{s}$ has rank $n=24p$ and
      and the product of $\coneName_{ex}$ has rank $n=3p$.
      }\label{fig:products}
\end{figure}
\begin{figure}
  \begin{tabular}{cc}
    \hspace{-1.3cm}  \includegraphics[width=.57\linewidth]{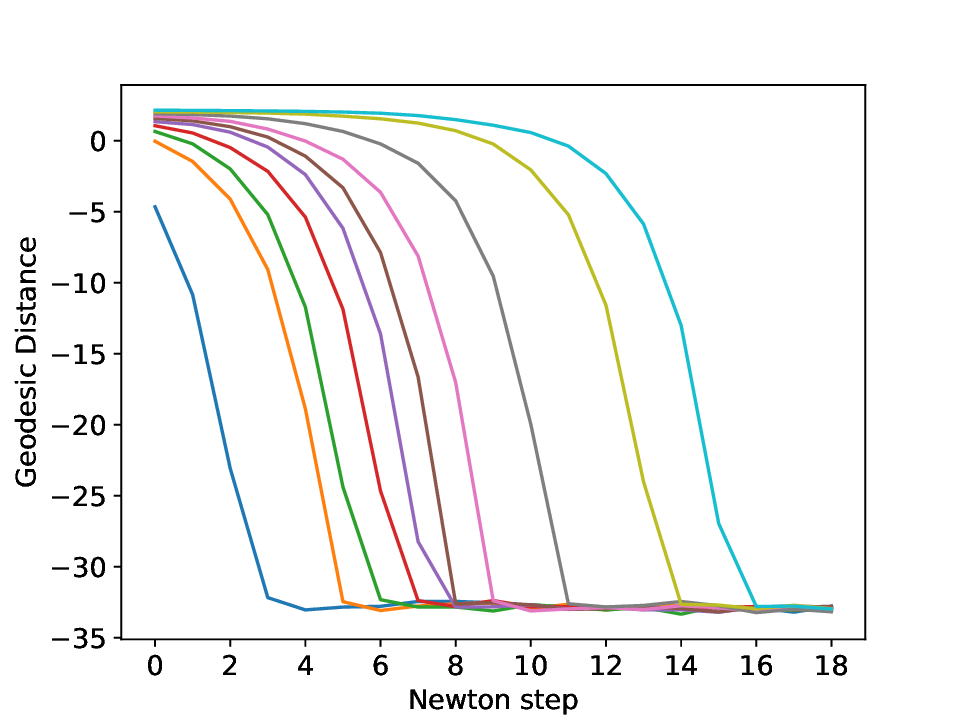}
    \includegraphics[width=.57\linewidth]{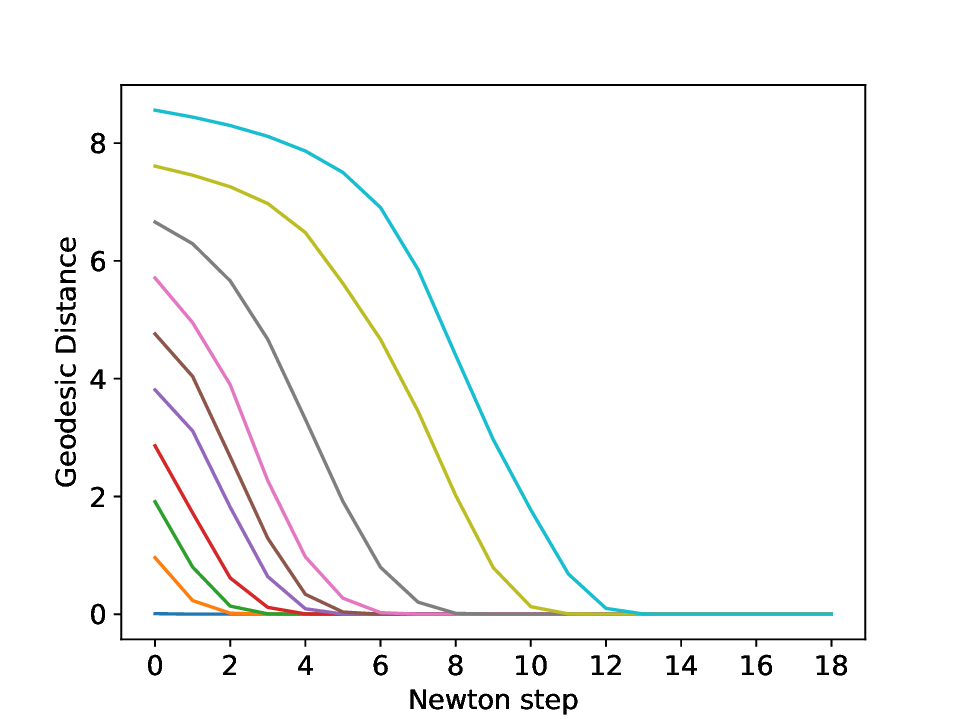}
  \end{tabular}
  \caption{Convergence of the centering procedure {\tt center} for different initialization 
  points in log (left) and linear (right) scalings. Plotted is geodesic distance to the centered point $\hat w(\mu)$. }~\label{fig:center} 
\end{figure}

\subsection{Products of special and exceptional cones}
For a Euclidean-Hurwitz algebra  $\mathbb{D}$, let $\mathbf{H}(\mathbb{D})^q$ 
denote the Hermitian matrices of order $q$ endowed 
with multiplication $X \circ Y = \frac{1}{2}(XY + YX)$
and inner-product $\trace X \circ Y$. If $\mathbb{D}$ is the real numbers $\mathbb{R}$, the complex numbers $\mathbb{C}$, or the quaternions $\mathbb{Q}$,
then $\mathbf{H}(\mathbb{D})$ is a special Jordan algebra of rank $q$.
When $\mathbb{D}$  is the octonions $\mathbb{O}$, 
then $\mathbf{H}(\mathbb{D})^3$ is an exceptional Jordan algebra of rank three.
(Note that in this notation,  $\mathbf{H}(\mathbb{R})^q$ refers to the
symmetric matrices $\mathbb{S}^q$.)

Denoting the cone-of-squares by $\mathbf{H}(\mathbb{D})^q_{+}$, we consider
the following cones
\[
  \coneName_{s} := \mathbf{H}(\mathbb{R})_+^{q} \times \mathbf{H}(\mathbb{C})_{+}^{q} \times \mathbf{H}(\mathbb{Q})_{+}^{q}, 
  \qquad \coneName_{ex} :=  \mathbf{H}(\mathbb{O})^3_{+}.
\]
Fixing $q = 8$, we plot (\Cref{fig:products}) the progress of {\tt longstep} on randomly generated instances formulated over $p$-fold products
$\coneName = \coneName_{s} \times \coneName_{s} \times \cdots \times \coneName_{s}$.
We similarly plot progress for products of $\coneName_{ex}$.
For the special cone $\coneName_{s}$, we used a divergence upper-bound of
$\beta = 100n$, where $n =  3 q p$, the rank of $\coneName$.  For the
exceptional cone $\coneName_{ex}$, we used a much tighter tolerance of $\beta =
\frac{1}{10} n$, where $n = 3p$.  This tighter
tolerance was necessary to avoid numerical errors we suspect are related to
errors in the power-series approximation of the geodesic
update~(\Cref{prop:geopower}). Recall use of this approximation is necessary
because $H(\mathbb{D})^3$ is not special  and hence prevents use of
the associative matrix exponential (\Cref{prop:geospecial}). Note that this
tighter tolerance leads to more iterations compared with the special cone $\coneName_s$, 
but also more regular updates of $\mu$ that are essentially independent of $n$.

\subsection{Global convergence}

The procedure {\tt center} used by {\tt longstep}  globally converges. That is, it always
returns  $\hat w(\mu)$ given an \emph{arbitrary} initial point $w_0 \in \inter
\coneName$ and centering parameter $\mu > 0$.
For a fixed problem instance, we plot convergence behavior for different 
initial conditions (\Cref{fig:center}).
We observe that convergence rate is divided into an initial and quadratic
phase.  These phases are expected from Theorem~\ref{thm:newton}.

\subsection{Implementation}
\newcommand\CC{C\nolinebreak[4]\hspace{-.05em}\raisebox{.4ex}{\relsize{-3}{\textbf{++}}}}
An implementation is available at ${\tt www.github.com/frankpermenter/conex}$.
All symmetric cones are directly supported,
including the Hermitian psd matrices with quaternion entries, the exceptional cone,
and generalized Lorentz cones of the form $\{ (x, t) \in \mathbb{R}^n \times \mathbb{R} : \langle x, x\rangle \le  t^2\}$
for arbitrary inner-products $\langle \cdot, \cdot \rangle$.
 \Cref{tab:results} compares performance with {\tt sdpt3} configured
to use Nesterov-Todd steps.
Computations  were performed in Ubuntu 18.04 on a machine with an
{\tt Intel Xeon CPU E5--2687W v4 @ 3.00GHz} processor. Both solvers were
configured to use the same linear algebra (BLAS) library.
For {\tt conex} error calculations,  $(x, s, y)$ was constructed from $w$ and
the final Newton step using Proposition~\ref{prop:feas}.
Results show that {\tt conex} achieves the same accuracy
in less time across a range of different problem sizes. 
We note that the relative timing difference is reduced for problems with $m > n$. 
For such problems, both solvers spend more  time on an identical calculation:
construction and solution of the linear system from~\Cref{sec:leastsquares};
see also \Cref{sec:ntcomp}.

\begin{table}
 % \hspace{-1cm}
  \smaller{}
  \begin{tabular}[t]{ccccccccccc}
    Parameters  & \multicolumn{2}{c}{Solver Time (sec)} & \multicolumn{2}{c}{Dual Residual} & \multicolumn{2}{c}{Duality Gap } & \multicolumn{2}{c}{$\lambda_{\min}(x)$} & \multicolumn{2}{c}{$\lambda_{\min}(s)$} \\
    $(n, m)$ & spdt3 &  conex  & spdt3 &  conex & sdpt3 & conex & sdpt3 & conex & sdpt3 & conex  \\ \toprule
%    (25,   63) & 6.0e-01 &  4.7e-02 &  1.2e-12 &  2.8e-12 & 1.1e-09 &  6.3e-10 \\ 
%    (50,   250)    & 2.4e+00 &  6.0e-01 &  4.0e-12 &  6.2e-12 & 5.1e-10 &  7.2e-10 \\ 
%    (100,   1000)  & 2.0e+01 &  1.4e+01 &  5.6e-11 &  2.6e-11 & 1.8e-09 &  7.0e-10 \\
%    \midrule
%    (25,   25) &  1.5e-01 &  8.3e-03 &  1.0e-10 &  5.4e-13 & 2.3e-10 &  1.2e-09 \\ 
%    (50,   50) &  8.5e-01 &  5.4e-02 &  2.4e-11 &  1.2e-12 & 3.1e-09 &  1.8e-09 \\ 
%    (100,   100) &   3.0e+00 &  8.5e-01 &  1.4e-10 &  2.8e-12 & 5.9e-09 &  2.5e-09 \\ 
(20, 20) & 1.1e-01 & 4.1e-03 &  1.4e-12 & 3.9e-12 & 1.4e-09 & 8.9e-10 & 3.2e-10 & 2.2e-10 & 3.2e-10 & 2.2e-10 \\ 
(50, 50) & 7.0e-01 & 1.1e-01 &  1.0e-12 & 1.5e-12 & 1.1e-09 & 1.9e-09 & 1.2e-10 & 2.0e-10 & 1.2e-10 & 2.0e-10 \\ 
(100, 100) & 3.1e+00 & 9.8e-01 &  2.0e-12 & 3.9e-12 & 9.7e-10 & 2.4e-09 & 7.6e-11 & 1.9e-10 & 7.6e-11 & 1.9e-10 \\ 
    \midrule
(20, 40) & 1.4e-01 & 1.6e-02 &  6.9e-11 & 7.7e-13 & 4.6e-10 & 7.2e-10 & 1.2e-10 & 1.8e-10 & 1.2e-10 & 1.8e-10 \\ 
(50, 250) & 1.8e+00 & 5.6e-01 &  1.5e-11 & 9.8e-12 & 5.3e-09 & 6.6e-10 & 1.5e-09 & 1.9e-10 & 1.5e-09 & 1.9e-10 \\ 
(100, 1000) & 1.9e+01 & 1.4e+01 &  3.4e-11 & 3.1e-11 & 6.5e-10 & 6.9e-10 & 1.7e-10 & 1.9e-10 & 1.7e-10 & 1.9e-10 \\ 
\end{tabular}
%  \begin{tabular}[t]{cccccc}
%    F & $\mu$ & GU. &   \\
% \hline
%   .3 & .5 & .6 \\ 
%    .3 & .5 & .6 \\ 
%    .3 & .5 & .6 \\ 
%\end{tabular}
  \caption{Solver time and residual comparison between our implementation {\tt conex} and {\tt sdpt3}. 
  The dual residual and duality gap refer to $k_1^{-1} \|A^*x-b\|$ and $k_2^{-1}|\langle c, x\rangle - b^{T}y|$
  for $k_1 = 1+\|b\|_{\infty}$ and $k_2 = 1+|\langle c, x\rangle| + |b^{T}y|$.
  Instances use a random $A : \mathbb{R}^m \rightarrow \mathbb{S}^n$ with
  $c := e$, $b:=A^*e$ and $\coneName = \mathbb{S}^n_{+}$.}~\label{tab:results}
\end{table}

\section*{Acknowledgements}
We thank Richard Y. Zhang and anonymous reviewers for helpful comments on an earlier draft.

{\small
\bibliographystyle{abbrvnat}
\bibliography{bib}
}

\appendix

\section{Appendix}\label{sec:appendix}
This section contains background results about the Euclidean Jordan algebra $\algName$ and
cone-of-squares $\coneName$ that we referenced without proof. The first
establishes properties of the quadratic representation $Q(u)v := 2 u\circ (u
\circ v) - (u\circ u) \circ v$.
\begin{lem}[\cite{faraut1994analysis}]~\label{lem:quad}
  The following statements hold.
  \begin{enumerate}
      %II.3.1.
    \item $Q(u)^{-1} = Q(u^{-1})$ for all invertible $u \in \algName$.
      %Proposition II.3.3
    \item\label{lem:quad:item:QAQinv} $(Q(u) v)^{-1} = Q(u^{-1}) v^{-1}$ for all invertible $u, v \in \algName$.
      %Proposition III.5.2
    \item $Q(Tu) = T Q(u) T^*$ for any $u \in \algName$ and automorphism $T: \algName \rightarrow \algName$ of $\coneName$,
    where $T^*: \algName \rightarrow \algName$ denotes the adjoint of $T$.
    \item $Q(u)^2 = Q(u^2)$ for all $u \in \algName$.
    \item $Q(u)e = u^2$ for all $u \in \algName$.
    \item $Q(u)$ is self-adjoint, i.e., $\langle Q(u) v, w\rangle = \langle v, Q(u) w\rangle$ for all $u, v, w \in \algName$.
  \end{enumerate}
  \begin{proof}
    The first properties are Propositions II.3.1., II.3.3, III.5.2, p. 55, and p. 48
    of~\cite{faraut1994analysis}. The last is evident from the definition
    of $Q(u)$ and the fact that Jordan multiplication is self-adjoint, i.e.,
    $\langle u \circ v, w \rangle = \langle  v, u \circ w \rangle$.
  \end{proof}
\end{lem}
\noindent The next establishes properties of orthogonal automorphisms of $\coneName$.
They trivially follow from the fact that such automorphisms are precisely
the Jordan-algebra automorphisms of $\algName$
given our use of the trace inner-product~\cite[p. 56]{faraut1994analysis}.
\begin{lem}\label{lem:auto}
  Let $M : \algName \rightarrow \algName$ be an orthogonal automorphism of $\coneName$.
  Then, the following statements hold for all $u \in \algName$.
  \begin{enumerate}
    \item If $u$ is an idempotent, i.e., $u \circ u = u$, then $Mu$ is an idempotent. 
    \item If $u$ has spectral decomposition $\sum^n_{i=1} \lambda_i e_i$,
     then $Mu$ has spectral decomposition $\sum^n_{i=1} \lambda_i Me_i$.
   \item~\label{lem:auto:exp}$\exp(Mu) = M \exp(u)$.
  \end{enumerate}
   Further, $Me = e$.
\begin{proof}
By use of the trace inner-product,  $M$ is also an automorphism of
  $\algName$~\cite[p. 56]{faraut1994analysis} and hence satisfies
  $(M x)\circ (My) = M (x \circ y)$.  Hence, $(M u) \circ (Mu) = M(u\circ u) = Mu$,
  showing the first statement.
    The second statement is immediate from the first:
    if $u$ has spectral decomposition $\sum^n_{i=1} \lambda_i e_i$,  then $Mu$ has
    decomposition $\sum^n_{i=1} \lambda_i Me_i$, since the $M e_i$
    are idempotent and pairwise orthogonal, i.e., 
    $\langle M e_i,  M e_j \rangle =  \langle  e_i, M^* M e_j \rangle =   \langle  e_i,  e_j \rangle =   0$.
    The  third is immediate from the second:
    \[
      \exp(Mu) = \sum^n_{i=1} \exp(\lambda_i) M e_i =\sum^n_{i=1} M \exp(\lambda_i)  e_i = M \exp(u).
    \]
    Finally, $Me = e$ given that $e = \exp(M 0) = M \exp(0) = Me$.
\end{proof}

\end{lem}
\end{document}